\documentclass[preprint, 3p]{elsarticle}

\usepackage[utf8]{inputenc}
\usepackage{tikz-cd} 
\usepackage{amsmath}
\usepackage{amsthm}
\usepackage{amsfonts}
\usepackage{subfig}
\usepackage{mathdots} 

 \usepackage{amssymb}
 \usepackage{amsthm}
 \usepackage{amsmath}
 \usepackage{tikz}
 \usepackage{graphicx}
 \usepackage{epstopdf}
 \usepackage{xurl}
 \usepackage{hyperref}
 \usepackage{multirow}
 \usetikzlibrary{patterns}

\newtheorem{theorem}{Theorem}[section]
\newtheorem{definition}[theorem]{Definition}
\newtheorem{lemma}[theorem]{Lemma}
\newtheorem{proposition}[theorem]{Proposition}
\newtheorem{corollary}[theorem]{Corollary}
\newtheorem{example}[theorem]{Example}

\newcommand{\im}{\textrm{Im}}
\newcommand{\Q}{\mathcal{Q}}
\newcommand{\A}{\mathcal{A}}
\newcommand{\B}{\mathcal{B}}
\newcommand{\C}{\mathcal{C}}

\newcommand{\R}{\mathcal{R}}
\newcommand{\PP}{\mathcal{P}}

\definecolor{myYellow}{rgb}{1.0, 0.87, 0.0}
\definecolor{myYtable}{rgb}{1.0, 0.75, 0.0}
\definecolor{umass}{rgb}{0.004, 0.3215, 0.5568}
\definecolor{etonblue}{rgb}{0.59, 0.78, 0.64}
 \definecolor{mycolor}{rgb}{0.12, 0.3, 0.17}
\definecolor{myblue}{rgb}{0.03, 0.27, 0.49}
\definecolor{mediumjunglegreen}{rgb}{0.11, 0.21, 0.18}
\definecolor{asparagus}{rgb}{0.53, 0.66, 0.42}
\definecolor{goldenpoppy}{rgb}{0.99, 0.76, 0.0}
\definecolor{lincolngreen}{rgb}{0.11, 0.35, 0.02}
\definecolor{red(ncs)}{rgb}{0.77, 0.01, 0.2}
\definecolor{darkorange}{rgb}{1.0, 0.55, 0.0}
\definecolor{safetyorange(blazeorange)}{rgb}{1.0, 0.4, 0.0}
\definecolor{smokeytopaz}{rgb}{0.58, 0.25, 0.03}
\definecolor{brightgreen}{rgb}{0.4, 1.0, 0.0}
\definecolor{tue}{rgb}{0.82, 0.14, 0.14}

\newcommand{\blue}[1]{{\color{black}#1}}

\journal{arXiv}

\begin{document}

\begin{frontmatter}

%% Title, authors and addresses

%% use the tnoteref command within \title for footnotes;
%% use the tnotetext command for theassociated footnote;
%% use the fnref command within \author or \address for footnotes;
%% use the fntext command for theassociated footnote;
%% use the corref command within \author for corresponding author footnotes;
%% use the cortext command for theassociated footnote;
%% use the ead command for the email address,
%% and the form \ead[url] for the home page:
%% \title{Title\tnoteref{label1}}
%% \tnotetext[label1]{}
%% \author{Name\corref{cor1}\fnref{label2}}
%% \ead{email address}
%% \ead[url]{home page}
%% \fntext[label2]{}
%% \cortext[cor1]{}
%% \address{Address\fnref{label3}}
%% \fntext[label3]{}

\title{Discrete Imprecise Copulas}

%% use optional labels to link authors explicitly to addresses:
%% \author[label1,label2]{}
%% \address[label1]{}
%% \address[label2]{}
\author[UL]{Toma\v{z} Ko\v{s}ir}
\author[TUe]{Elisa Perrone\footnote{Corresponding author, email \url{e.perrone@tue.nl}}}

\address[UL]{Faculty of Mathematics and Physics, University of Ljubljana, Jadranska 19, 1000 Ljubljana, Slovenia}
\address[TUe]{Department of Mathematics and Computer Science, Eindhoven University of Technology, Groene Loper 5, 5612 AZ, Eindhoven, The Netherlands}

\begin{abstract}
In this paper, we study discrete quasi-copulas associated with imprecise copulas. We focus on discrete imprecise copulas that are in correspondence with the Alternating Sign Matrices and provide some construction techniques of dual pairs. Additionally, we show how to obtain (full-domain) self-dual imprecise copulas through patchwork techniques.
We further investigate the properties of the constructed dual pairs and prove that they are coherent and \blue{avoid} sure loss.
\end{abstract}

%%%%%%KEYWORD%%%%%%%
\begin{keyword}
Imprecise copulas, Discrete quasi-copulas, Duality, Alternating Sign Matrices, Patchwork.
%% keywords here, in the form: keyword \sep keyword
%% PACS codes here, in the form: \PACS code \sep code
%% MSC codes here, in the form: \MSC code \sep code
%% or \MSC[2008] code \sep code (2000 is the default)
\end{keyword}

\end{frontmatter}

\section{Introduction}
\label{sec:1}
In many research fields, such as decision theory \cite{JSA,MMM,Troff}, reliability theory \cite{Coolen,OKS,UtCo,YDSS}, financial risk management \cite{ADEH,Delbaen,Vicig}, and game theory \cite{MiMo,Nau}, the assessment of the exact probability is either not possible due to lack of information or not desirable due to personal beliefs that obscure the exact probability value. Such situations require more flexible models and can be tackled through the theory of imprecise probability initiated by Walley \cite{Walley} (see \cite{ACdCT} for a more recent introduction to the topic). 

\smallskip
Imprecise probability of an event $A$ may be described as an interval $[\underline{P}(A),\overline{P}(A)]$ containing possible values of the exact probability $P(A)$. The lower probability $\underline{P}(A)$ captures the evidence supporting $A$, while the upper probability $\overline{P}(A)$ expresses the lack of evidence against $A$ (here $\underline{P}$ and $\overline{P}$ need not be \blue{additive} measures but are assumed to be monotone in sets). 
Modeling imprecision of random variables is based on the notion of a \textit{$p$-box} (a probability box) \cite{FKGMS,FeTu,PVMM-2,TroDe,TMD,UtDe}. For a single random variable $X$, a $p$-box is a pair of cumulative distribution functions $(\underline{F},\overline{F})$ such that $\underline{F}(x)\leq \overline{F}(x)$ for all $x\in\mathbb{R}$, where $\underline{F}(x)$ and $\overline{F}(x)$ represent the lower and upper bound of the probability of the event $[X\leq x]$. The point-wise infimum and supremum of any $p$-box, also called \textit{envelopes}, are exactly its lower and upper bound. Thus, they remain within the class of cumulative distribution functions. 
The situation changes dramatically for distribution functions of two or more variables, where the envelopes of a $p$-box may not be cumulative distribution functions \cite{PVMM-2}. 
To deal with this limitation, Troffaes and Destercke, in their model \cite{TroDe},  assume that the bounds of a multivariate $p$-box are multivariate distribution functions. Another approach to tackle this issue is based on a generalization of the theory of \textit{copulas} to the imprecise setting.

\smallskip
The theoretical foundation of \textit{copulas} is Sklar's Theorem \cite{Sklar} stating that copulas are joint probability distributions with uniform margins in $[0,1]$. 
As such, copulas play a central role when modeling dependence between random variables (see, e.g., \cite{DurSem15,joe_2014,Nelsen}). 
In \cite{NeUbF}, the authors show that the Dedekind-MacNeille completion of the poset of bivariate copulas for the point-wise order is the set of more general functions called \textit{quasi-copulas} \blue{\cite{ANS,A-GMDeB,genest_characterization_1999,NQRU-2,NQRU}} (see \cite{OmSt-4} for multivariate version of this result). Hence, the envelopes of a bivariate $p$-box of copulas are, in general, quasi-copulas that satisfy some additional conditions \cite{PVMM}.

\smallskip
The connection between a $p$-box of copulas and quasi-copulas makes it natural to define an \textit{imprecise copula} as a pair of quasi-copulas $(\underline{Q},\overline{Q})$ such that $\underline{Q}(x,y)\leq \overline{Q}(x,y)$ for all $x,y\in [0,1]$, satisfying some additional conditions \cite{PVMM}. An imprecise copula $(\underline{Q},\overline{Q})$ is also considered as the interval $[\underline{Q},\overline{Q}]$ in the set of all (quasi)-copulas with respect to the point-wise order. In search of an appropriate definition of imprecise copula, one wishes to generalize Sklar's Theorem to the imprecise setup. We refer to \cite{MMPV,FullSklar} for such a generalization of Sklar's Theorem in the bivariate case and to \cite{MultiSklar} for the multivariate case. See also the introduction of \cite{Stop} for an extended discussion of the topic. Recently, imprecise copulas for shock models were studied in \cite{DK-BOS,OmSk}.
Imprecise copulas can also be obtained through transformations of quasi-copulas. For example, two of the transformations introduced by Dibala et al. in \cite{DS-PMK}, i.e., the \textit{main} and the \textit{opposite defect} functions, give two imprecise copulas when applied to a quasi-copula (see \cite[Thm. 7]{FinalS}). 
%Dibala, Saminger-Platz, Mesiar and Klement \cite{DS-PMK} introduce \textit{defect} functions for quasi-copulas and show that related transformations map quasi-copulas to quasi-copulas. 
%In particular, two of the transformations they introduce, i.e., the \textit{main} and the \textit{opposite}, give two imprecise copulas when applied to quasi-copulas \cite[Thm. 7]{FinalS}. 
Iterating the two transformations leads to a so-called \textit{self-dual imprecise copula} in the limit \cite[Prop. 10(c)]{FinalS}. 

\smallskip
The notion of defect presented in \cite{DS-PMK} has also been key in answering an open question posed in \cite{MMPV}, that is, if any imprecise copula contains a copula. 
This question relates to two important notions in the theory of imprecise probability, that \blue{is when a $p$-box  \emph{avoids sure loss} and  when it is \emph{coherent}} \cite{ACdCT,Walley}. As in \cite{PVMM-2}, for an imprecise copula $(\underline{Q},\overline{Q})$ we denote by $\C$ the set of all proper bivariate copulas $C$ such that $\underline{Q}(x,y)\leq C(x,y)\leq \overline{Q}(x,y)$ for all $x,y\in [0,1]$. Then an imprecise copula $(\underline{Q},\overline{Q})$  \blue{\emph{avoids sure loss}} if $\C\neq\emptyset$ \cite[Prop. 6]{PVMM-2} and is \emph{coherent} if the point-wise infimum of $\C$ is equal to $\underline{Q}$ and the point-wise supremum of $\C$ is equal to $\overline{Q}$ \cite[Prop. 9]{PVMM-2}. 
Omladi\v c and Stopar \cite[Example 18]{FinalS} provide an example of a bivariate imprecise copula that \blue{does not avoid} sure loss and so it is not coherent either. 
They also characterize imprecise copulas that \blue{avoid} sure loss (see \cite[Prop. 16]{FinalS} for the discrete case and \cite[Thm. 17]{FinalS} for the general case), and in doing so, they give a negative answer to the question posed in \cite{MMPV}.

\smallskip
Some of the recent developments in the literature of imprecise copulas rely on discrete versions of such functions. For example, in \cite[\S 5]{K-BKOS}, the authors provide an example of a self-dual imprecise copula on the unit square by extending an imprecise copula defined on a grid domain (see also \cite[Example 11]{FinalS}). 
This shows the benefits of investigating the properties of imprecise copulas in a discrete setting.
Discrete copulas and discrete quasi-copulas have been extensively studied in terms of their matrix representations \cite{AST-1,F-SQ-MU-F,KolMor,MST,Q-MS}. Their spaces are \textit{convex polytopes}, i.e., compact convex spaces with a finite number of extreme points \cite{Ziegler_95}. 
Namely, for uniform grid domains, discrete copulas correspond to the Birkhoff polytopes \cite{Ziegler_95}, while the discrete quasi-copulas to the Alternating Sign Matrix (ASM) polytopes \cite{striker_2009} (see \cite{Perrone2022, Perrone2021, PSU} and references therein for more details on the topic). 
%Such polytopes, respectively, have the permutation matrices and the alternating sign matrices as their extreme points \cite{striker_2009} \textcolor{red}{add citation on the Birkhoff polytope}. 
Despite these results, no research has been done to specifically analyze the properties of \textit{discrete imprecise copulas}, e.g., restricted to a uniform grid domain, especially those that are self-dual.
In this paper, we fill this gap and derive interesting results on self-duality for discrete imprecise copulas. We first show how to construct full-domain (self-dual) imprecise copulas through patchwork techniques. 
Then, we give several constructions of self-dual pairs of discrete imprecise copulas and study \blue{when they are coherent and when they avoid sure loss.}

\smallskip
The paper is structured as follows. In Section \ref{sec:2]} we review basic definitions of copulas, quasi-copulas and their defects, imprecise copulas and self-dual imprecise copulas. In Section \ref{sec:3} we show that a patchwork of imprecise copulas is an imprecise copula, while self-duality is also preserved by patchwork techniques. These results hold for both the continuous and discrete cases. In Section \ref{sec:4} we focus our attention on special types of Alternating Sign Matrices (ASM) that are dense. We prove how irreducible dense ASM are patched together to get a general dense ASM. We use this result to construct a maximal chain of self-dual imprecise copulas that all correspond to dense ASM. In addition, we prove that all the dual pairs in the chain are coherent imprecise copulas. 
In Section \ref{isolated pairs}, we construct non-dense self-dual imprecise copulas in dimensions $n\ge 7$ that are based on Example 11 of \cite{FinalS}. We show that these imprecise copulas are also coherent. We conclude with a short overview and some suggestions for future research.

\section{Discrete quasi-copulas, their defects, and discrete imprecise copulas}
\label{sec:2]}

In this section, we present preliminaries on copulas, quasi-copulas, and imprecise probability that we use throughout the paper. 
We first recall the definitions of full-domain copulas and quasi-copulas. 
Then, we move to the discrete setting and discuss polytopal representations of discrete (quasi-) copulas, i.e., their properties in terms of associated matrices. 
Finally, we introduce the notions of defects and imprecise copulas in the discrete framework.
The following defines copulas as functions that satisfy functional inequalities \cite{Nelsen}.

\begin{definition}
\label{Def:cop}
A function $C:[0,1]^2 \rightarrow [0,1]$ is a copula if and only if
\begin{enumerate}
\item[(C1)] $C(x,0)=C(0,y)=0$ and $C(x,1)=x$, $C(1,y)=y$ for every $x,y \in [0,1]$;
    \vspace{3pt}
\item[(C2)] $C(x_1,y_1) + C(x_2,y_2) \geq C(x_1,y_2) + C(x_2,y_1)$ for every $x_1, x_2, y_1, y_2 \in [0,1]$ such that $x_1 \leq x_2, y_1 \leq y_2$.
\end{enumerate}
\end{definition}

Hence, bivariate copulas are functions defined on the unit square that are uniform on the boundary (C1) and supermodular (C2). The bivariate copulas form a poset $\PP$ through the point-wise partial order, that is $C\prec C^\prime$ whenever $C(x,y) \leq C^\prime(x,y)$ for all $(x,y)\in[0,1]^2$ (see Definition~2.8.1 and Example~5.13 in \cite{Nelsen}). However, not all copula pairs, $C$ and $C^\prime$, have both a least upper bound and greatest lower bound with respect to $\prec$ in the set of all copulas. Therefore, $\PP$ is not a lattice. The functions that complete $\PP$ to a lattice under $\prec$ are the so-called \textit{quasi-copulas} \cite{NeUbF}, and are defined by Genest et al.~\cite{genest_characterization_1999} as follows.

\begin{definition}
\label{Def:quasi-copula}
A function $Q:[0,1]^2 \rightarrow [0,1]$ is a quasi-copula if and only if it satisfies the following conditions:
\begin{enumerate}
    \item[(Q1)] $Q(x,0)=Q(0,y)=0$ and $Q(x,1)=x$, $Q(1,y)=y$ for every $x,y\in [0,1]$,
    \item[(Q2)] $Q$ is increasing in each component.
    \vspace{3pt}
    \item[(Q3)] $Q$ satisfies the 1-Lipschitz condition, i.e., for every $x_1,x_2,y_1,y_2 \in [0,1],$ 
    $ |Q(x_2,y_2) -Q(x_1,y_1)| \leq |x_2 - x_1| + |y_2 - y_1|.$
\end{enumerate}
\end{definition}

Equivalently, the authors of \cite{genest_characterization_1999} characterize quasi-copulas as functions that satisfy the boundary condition (C1) and are supermodular on any rectangle with at least one edge on the boundary of the unit square. Namely, we have the following.

\begin{proposition}
\label{prop:char-qc}
A function $Q:[0,1]^2\to [0,1]$ is a \emph{quasi-copula} if \blue{and only if}
\begin{enumerate}
    \item[(Q1)] $Q(x,0)=Q(0,y)=0$ and $Q(x,1)=x$, $Q(1,y)=y$ for every $x,y\in [0,1]$,
    \item[(Q2b)] $Q$ is supermodular on any rectangle with at least one edge on the boundary of the unit square, i.e., 
    \[Q(x_1,y_1) + Q(x_2,y_2) \geq Q(x_1,y_2) + Q(x_2,y_1)\]
    whenever $x_1\leq x_2$, $y_1\leq y_2$ and at least one of $x_1$, $x_2$, $y_1$, $y_2$ is either equal to 0 or to 1. 
\end{enumerate}
\end{proposition}

Conditions (Q2b) of Proposition~\ref{prop:char-qc} are a subset of the inequalities of (C2) in Definition~\ref{Def:cop}. Thus, any copula is also a quasi-copula.
Given a (quasi-) copula $Q$, we define the $Q$-volume of a rectangle $R=[x_1,x_2]\times[y_1,y_2]\subset [0,1]^2$ with $x_1\leq x_2$, $y_1\leq y_2$ by 
\begin{equation}\label{V_Q}
    V_Q(R)=Q(x_1,y_1) + Q(x_2,y_2) - Q(x_1,y_2) - Q(x_2,y_1).
\end{equation}
We say that a quasi-copula $Q$ is \textit{proper} if there exists a rectangle $R=[x_1,y_1] \times [x_2,y_2] \subset [0,1]^2$ such that the volume induced by $Q$ is negative, which is $V_Q(R) = Q(x_1,y_1) + Q(x_2,y_2) - Q(x_1,y_2) - Q(x_2,y_1) < 0$. 

We notice that Proposition~\ref{prop:char-qc} and Definition~\ref{Def:cop} characterize quasi-copulas and copulas through the mass that they can associate with each rectangle of the unit square $V_Q(R)$. In the next subsection, we see how the mass distribution plays a role in connecting discrete (quasi-) copulas with special matrices.

\subsection{Polytopes of discrete copulas and quasi-copulas}

We now recall the most relevant results on copulas and quasi-copulas in the discrete setting. The general case of discrete copulas and quasi-copulas on arbitrary grid domains and their polytopal representations has been tackled in \cite{AST-1,AST-2,KMMS,MST,Perrone2022,Perrone2021, PSU, Q-MS}. 
Here we restrict ourselves to the special case of uniform grid domains, that is, we consider a uniform partition of the unit interval $I_n=\{ 0, 1/n, \ldots, (n-1)/n, 1 \}$ for $n\in\mathbb{N}$ and define discrete (quasi-) copulas on the square grid $I_n^2$.
In such a case, a discrete copula is a function $C: I_n^2 \rightarrow [0,1]$ that satisfies the conditions (C1) and (C2) of Definition~\ref{Def:cop}. Analogously, we define discrete quasi-copulas as functions on $I_n^2$ satisfying (Q1) and (Q2b) of Proposition~\ref{prop:char-qc}. 

\smallskip
We denote by $\Q_n$ the set of $n \times n$ discrete quasi-copulas and, respectively, by $\C_n$ the set of all discrete copulas on $I_n^2$. 
In the following, we mostly refer to the discrete quasi-copulas since discrete copulas are their special cases.
We notice that any discrete quasi-copula $Q: I_n^2 \rightarrow [0,1]$ can be expressed as a map $Q':L_n^2 \rightarrow \blue{[0,n]}$, with $L_n=\{0,1,\ldots,n\}$, through the following transformation:
\begin{equation}
\label{eq:Qprime}
Q'(r,s)=n\cdot Q\left(\frac{r}{n},\frac{s}{n}\right) \textrm{ for } r,s\in L_n.    
\end{equation}
Hence, we denote by $\Q_n'$ the set of all functions on $L_n^2$ that correspond to discrete quasi-copulas and by $\C_n'$ the set of all functions on $L_n^2$ that correspond to discrete copulas. Functions in $\C'_n$ and $\Q'_n$ satisfy conditions on $L_n$ that are analogous to conditions $(C1)$-$(C2)$ and $(Q1)$-$(Q3)$. \blue{We include precise definitions for clarity.%We denote these conditions by $(C1')$-$(C2')$ and $(Q1')$-$(Q3')$.
\begin{definition}
\label{Def:discete_cop}
A function $C:L_n^2 \rightarrow [0,n]$ is a  \emph{discrete copula} if and only if
\begin{enumerate}
\item[(C1')] $C(r,0)=C(0,s)=0$ and $C(r,n)=r$, $C(n,s)=s$ for every $r,s \in L_n$;
    \vspace{3pt}
\item[(C2')] $C(r_1,s_1) + C(r_2,s_2) \geq C(r_1,s_2) + C(r_2,s_1)$ for every $r_1, r_2, s_1, s_2 \in L_n $ such that $r_1 \leq r_2, s_1 \leq s_2$.
\end{enumerate}
\end{definition}

\begin{definition}
\label{Def:discrete quasi-copula}
A function $Q:L_n^2 \rightarrow [0,n]$ is a \emph{discrete quasi-copula} if and only if it satisfies the following conditions:
\begin{enumerate}
    \item[(Q1')] $Q(r,0)=Q(0,s)=0$ and $Q(r,n)=r$, $Q(n,s)=s$ for every $r,s\in L_n$,
    \item[(Q2')] $Q$ is increasing in each component.
    \vspace{3pt}
    \item[(Q3')] $Q$ satisfies the 1-Lipschitz condition, i.e., for every $r_1,r_2,s_1,s_2 \in L_n,$ 
    $ |Q(r_2,s_2) -Q(r_1,s_1)| \leq |r_2 - r_1| + |s_2 - s_1|.$
\end{enumerate}
\end{definition}

Note that Proposition \ref{prop:char-qc} also has its discrete counterpart. For future reference we only state the discrete counterpart of Condition (Q2b).\emph{ 
\begin{enumerate}
    \item[(Q2b')] $Q:L_n^2\to [0,n]$ is supermodular on any rectangle with at least one edge on the boundary, i.e., 
    \[Q(r_1,s_1) + Q(r_2,s_2) \geq Q(r_1,s_2) + Q(r_2,s_1)\]
    whenever $r_1\leq r_2$, $s_1\leq s_2$ and at least one of $r_1$, $r_2$, $s_1$, $s_2$ is either equal to $0$ or to $n$. 
\end{enumerate}
}
}

Quasi-copulas and copulas on $I_n$ (or on $L_n$) can naturally be identified with square matrices of size $n$. 
In particular, we denote the matrix of $Q\in\Q_n$ by the same letter $Q$ and write  
$$Q=\begin{pmatrix}
 q_{11} & q_{12} & \cdots  & q_{1n} \\
 q_{21} & q_{22} & \cdots  & q_{2n} \\
 \vdots & \vdots & & \vdots\\
 q_{n1} & q_{n2} & \cdots  & q_{nn} 
\end{pmatrix},\ \ \mathrm{where}\ q_{rs}=Q\left(\frac{r}{n},\frac{s}{n}\right).$$
We here omit the values at $r=0$ and $s=0$ that are all $0$. In addition, we preserve the usual directions of the rows and columns in the matrices, which correspond to the usual coordinate system that is rotated by $90^{\circ}$ in a positive direction.

% Similarly, for a discrete copula $C \in \C$, we have 
% $$C=\left[\begin{array}{cccc}
%  c_{11} & c_{12} & \cdots  & c_{1n} \\
%  c_{21} & c_{22} & \cdots  & c_{2n} \\
%  \vdots & \vdots & & \vdots\\
%  c_{n1} & c_{n2} & \cdots  & c_{nn} 
% \end{array}\right],\ \ \mathrm{where}\ c_{rs}=C(r,s).$$

\smallskip
As mentioned in the introduction, the sets $\C_n$ and $\Q_n$ are in one-to-one correspondence with two convex polytopes studied in discrete geometry. That is, we have the following connections.
\begin{description}
\item[]\textbf{Discrete copulas and Birkhoff polytopes.} The set of discrete copulas on $I_n^2$ \blue{(or $L_n^2$)} is in one-to-one correspondence with the \textit{Birkhoff Polytope} $\B_n$ (see Proposition~2 in \cite{KMMS}). This polytope is the space of the so-called \textit{bistochastic matrices} (BM), i.e., square matrices with non-negative entries and row and column sums equal to one \cite{Ziegler_95}.
The extreme points, i.e., the \textit{vertices}, of $\B_n$ are the permutation matrices, which are bistochastic matrices with entries equal to $0$ or $1$.
\item[] \textbf{Discrete quasi-copulas and Alternating Sign Matrix Polytopes.} The set of discrete quasi-copulas on $I_n^2$ corresponds to the \textit{Alternating Sign Matrix Polytope} $\A^{B}_n$ \cite{AST-1}. This polytope has been studied in \cite{striker_2009} and is the space of the \textit{alternating bistochastic matrices} (ABM), i.e., square matrices that generalize the bistochastic matrices allowing for negative entries. Specifically, Theorem~2.1 of \cite{striker_2009} states that a matrix $A=[a_{ij}]_{i,j=1}^n$ belongs to the Alternating Sign Matrix Polytope, i.e., $A$ is an \emph{ABM} if, for all \blue{$s, r =1, \ldots, n$}, it satisfies the following inequalities:
\begin{equation}\label{ASM_conditions}
\blue{0\leq\sum_{i=1}^{r}a_{i,s}\leq 1,
\quad
0\leq\sum_{j=1}^{s}a_{r,j}\leq 1, \quad \sum_{i=1}^na_{i,s}=1, \quad
\sum_{j=1}^na_{r,j}=1.}
\end{equation}
The vertices of $\A^{B}_n$ are the alternating sign matrices (ASM), which are alternating bistochastic matrices with entries equal to $-1, 0,$ or $1$. Conditions \blue{in Equation} (\ref{ASM_conditions}) imply that in each row and each column of ASM the nonzero entries alternate in sign. We denote the set of $n \times n$ alternating sign matrices by $\A_n$. By definition, in each row and column of an ASM the first and last non-zero entries are equal to $1$. 
 Additionally, all permutation matrices are ASM. We say that $A\in\A_n$ is a \emph{proper ASM} if it contains at least one entry equal to $-1$.
We say that a matrix is \textit{nonnegative} if all its entries are nonnegative.
\end{description}

\noindent
Given $A\in\A^B_n$, we can construct the unique matrix $Q\in \Q_n$ associated with $A$ as in \cite[Prop. 3 \& Cor. 1]{AST-1}. Namely, we proceed as follows
\begin{equation}
\label{eq:bij1}
Q\left(\frac{r}{n},0\right)=Q\left(0,\frac{s}{n}\right)=0,\ \textrm{ for } r,s\in L_n,\textrm{ and }
Q\left(\frac{r}{n},\frac{s}{n}\right)=\frac{1}{n}\sum_{i=1}^r\sum_{j=1}^s a_{i,j}, \textrm{ for }r,s\ge 1.
\end{equation}
\blue{In a similar way, we can derive an element $Q\in\Q'_n$ associated to $A$ by 
\begin{equation}
\label{eq:bij1_prime}
Q\left({r},0\right)=Q\left(0,{s}\right)=0,\ \textrm{ for } r,s\in L_n,\textrm{ and }
Q\left({r},{s}\right)=\sum_{i=1}^r\sum_{j=1}^s a_{i,j}, \textrm{ for }r,s\ge 1.
\end{equation}
}

Conversely, for any $Q\in\blue{\Q'_n}$, the uniquely associated matrix $A=\left[a_{st}\right]_{s,t=1}^n $ in $\A^B_n$ is given by
\begin{equation}
\label{eq:transf}
a_{s,t}=Q(s,t)+Q(s-1,t-1)-Q(s,t-1)-Q(s-1,t).
\end{equation}
We sometimes write $A=A(Q)$ and $Q=Q(A)$ to emphasize the correspondence between $A$ and $Q$. The one-to-one relation can be transferred to $\Q_n$ by \blue{dividing} all values by $n$. 
Furthermore, the same bijective maps link the elements of \blue{$\C'_n$} or $\C_n$  with the elements of $\B_n$.
\blue{Note that the (discrete) Fr\'echet upper and lower bounds $M_n$ and $W_n$, respectively, correspond to the matrices 
$$I_n=\begin{pmatrix}
  1 & 0 & 0 & \cdots  & 0 & 0\\
  0 & 1 & 0 & \ddots  & 0 & 0\\
  0 & 0 & 1 &\ddots  & 0 & 0\\
 \vdots & \ddots & & \ddots & \ddots & \vdots\\
 0 & 0 & 0&  \ddots & 1 & 0\\
 0 & 0 & 0&  \cdots & 0 & 1
\end{pmatrix}, \text { and } H_n=\begin{pmatrix}
  0 & 0 &  \cdots  & 0 & 0 & 1\\
  0 & 0 &  \iddots  & 0 & 1 & 0\\
  0 & 0 & \iddots  & 1 & 0 & 0\\
  \vdots  & \iddots & \iddots & \iddots  & & \vdots\\
  0 & 1 &  \iddots & 0 & 0 & 0 \\
  1 & 0 &   \cdots & 0 & 0 &0
\end{pmatrix},$$
which are the identity matrix and the anti-identity matrix, respectively.

In the discrete setup, we mostly use the correspondence between $\A^{B}_n$ and $\Q'_n$ and its restriction to the correspondence between $\B_n$ and $\C'_n$.
}

Transformations $A\mapsto Q(A)$ and $Q\mapsto A(Q)$ can be expressed by matrix multiplications. More precisely, denote by $J_n$ the the nilpotent upper-triangular Jordan matrix of size $n$:
$$J_n=\begin{pmatrix}
 0 & 1 & 0 & \cdots  & 0 & 0\\
 0 & 0 & 1 & \ddots  & 0 & 0\\
 0 & 0 & 0 & \ddots  & 0 & 0\\
 \vdots & \vdots & & \ddots & \ddots & \vdots\\
 0 & 0 & 0&  \cdots & 0 & 1 \\
 0 & 0 & 0&  \cdots & 0 & 0 
\end{pmatrix},$$
write $K_n=I_n-J_n$ and $A^T$ for the transpose of a matrix $A$. Then \blue{for $Q\in\Q_n'$ we have}
$$A(Q)=K_n^TQK_n \text{ and } Q(A)=(K_n^T)^{-1}AK_n^{-1}.$$

Next, we consider the set $\R$ of all rectangles with vertices on the grid $L_n\times L_n$. So, $\R$ is the set of all rectangles $[i,j]\times [k,l]$, where $0\leq i\leq j\leq n$ and $0\leq k\leq l\leq n$, and $[i,j]=\{i,i+1,\ldots,j\}$ for $i\leq j$.
We also notice that the corresponding ABM $A=\left[a_{st}\right]_{s,t=1}^n $ to a discrete quasi-copula $Q \blue{\in\Q_n'}$ represents the values of the volumes of $Q$ for each square of the grid. Indeed, we have that for any rectangle $R=[i,j]\times[k,l]\in\R$ 
\begin{equation}\label{V_Q(A)}
    V_Q(R)=\sum_{r=i+1}^{j}\sum_{s=k+1}^{l}a_{r,s}.
\end{equation}

The extreme points of $\A^B_n$ and $\B_n$ play a special role in their copula counterparts. 
In particular, they are linked with discrete quasi-copulas and discrete copulas whose range is exactly equal to $I_n$. Since $I_n$ is always contained in the range of $Q \in \Q_n$, we call discrete quasi-copulas $Q$ with $\im(Q)=I_n$ \emph{of minimal range} (analogously for discrete copulas). Such discrete quasi-copulas are called irreducible in \cite[Def. 8]{AST-1} or \cite[Def. 2.3]{Kob} because they are in fact the extreme points of $\Q_n$.
%\smallskip

%We denote by $\Q_m$ the set of all discrete quasi-copulas of minimal range and by $\Q_m'$ the set of all corresponding functions $Q'$ on $L_n^2$.
The simplest proper \blue{discrete} quasi-copula $Q\in\Q'_n$ of minimal range and its corresponding ASM are
    \begin{equation}
        \label{3x3 proper ASM}
Q=\begin{pmatrix}
       0  & 1 & 1 \\
       1  & 1 & 2 \\
       1  & 2 & 3
\end{pmatrix} \mathrm{and}\ 
    A(Q)=\begin{pmatrix}
        0 & 1 & 0 \\
        1 & -1 & 1 \\
        0 & 1 & 0
\end{pmatrix}
    \end{equation}
\blue{Due to Condition (Q2b') there is no proper discrete quasi-copula on  $L_2^2$ since there is no rectangle on the grid that would not have an edge on the boundary. Thus, the smallest proper discrete quasi-copula is on $L_3^2$. This grid contains only one square that does not touch the boundary. So, the discrete quasi-copula in Eq.~\eqref{3x3 proper ASM} is the only proper discrete quasi-copula of minimal range on $L_3^2$.}

\medskip

In the following, we study the existence of a (discrete) copula $C$ such that \blue{we have $Q_1\leq C\leq Q_2$ for two quasi-copulas $Q_1$ and $Q_2$ with respect to the point-wise order}. To do so, we need to introduce the notions of \textit{defect} and of \textit{imprecise copula}.

\subsection{Defects of discrete quasi-copulas}
Defects of quasi-copulas are introduced in Dibala et al \cite[\S 3]{DS-PMK}. 
Here we define them only in discrete settings. Since we prefer the usual directions for the numbering of rows and columns in matrices, all our arrows are rotated by $90^{\circ}$ as compared to the arrows in \cite{DS-PMK}. 

We consider again the set $\R$ of all rectangles on the grid $L_n^2$. For each pair $(r,s)\in L_n^2$ we introduce four sets of rectangles. These are rectangles that have one of the corners fixed on the grid. In particular, we define four sets of rectangles:
\begin{eqnarray*}
    \R_{\searrow}(r,s)&=&\left\{[r+1,r+i]\times [s+1,s+j];\ 1\leq i\leq n-r, 1\leq j\leq n-s\right\},\\
    \R_{\swarrow}(r,s)&=&\left\{[r+1,r+i]\times [j,s];\ 1\leq i\leq n-r, 1\leq j\leq s\right\},\\
    \R_{\nwarrow}(r,s)&=&\left\{[i,r]\times [j,s];\ 1\leq i\leq r, 1\leq j\leq s\right\},\\
    \R_{\nearrow}(r,s)&=&\left\{[i,r]\times [s+1,s+j];\ 1\leq i\leq r, 1\leq j\leq n-s\right\}.
\end{eqnarray*}

\blue{The choice of the rectangles in the four sets requires some explanation. Let us consider an ABM $A$ and the corresponding discrete quasi-copula  $Q(A)$. Then, we observe that the entry $a_{rs}$ of $A$ gives the volume of $[r-1,r]\times[s-1,s]$ by Equation~\eqref{eq:transf}. In addition, the entry $q_{rs}$ of $Q(A)$ is the volume of the rectangle $[0,r]\times [0,s]$. The square $[r-1,r]\times[s-1,s]$ is included in any non-degenerated rectangle in the $\nwarrow$ direction with respect to the grid point $(r,s)$ and each rectangle of $\R_{\nwarrow}(r,s)$ is contained in $[0,r]\times [0,s]$. This rectangle has only degenerated intersection with rectangles in the other three directions with respect to the grid point $(r,s)$. Thus the defect directions are chosen with respect to the points of the grid $L_n^2$, while the values of $A$ or $Q(A)$ give volumes of particular squares or rectangles in the $\nwarrow$ direction from the chosen grid point. }

Now, for each discrete quasi-copula $Q\in\Q_n'$ we define four \emph{directional defect matrices} $D_{\nearrow}^Q,D_{\nwarrow}^Q,D_{\swarrow}^Q$ and $D_{\searrow}^Q$.  Their entries are numbered by the \blue{grid points in the set $L_n^2$ as explained above.} We omit the bottom row and left column of zeros and consider them as square matrices of size $n$. They are given by
\begin{eqnarray*}
    D_{\searrow}^Q(r,s)&=&\min\{0,V_Q(R);\ R\in\R_{\searrow}(r,s)\},\\
    D_{\swarrow}^Q(r,s)&=&\min\{0,V_Q(R);\ R\in\R_{\swarrow}(r,s)\},\\
    D_{\nwarrow}^Q(r,s)&=&\min\{0,V_Q(R);\ R\in\R_{\nwarrow}(r,s)\},\\
    D_{\nearrow}^Q(r,s)&=&\min\{0,V_Q(R);\ R\in\R_{\nearrow}(r,s)\}.
\end{eqnarray*}
Observe that if $Q$ is a \blue{(discrete)} copula then all the defects are equal to $0$. So, it is only interesting to study defects for proper \blue{(discrete)} quasi-copulas. 

Two additional defect matrices are important. They are called \emph{the main and the opposite defect matrices} and denoted by $D_M^Q$ and $D_O^Q$, respectively. Their entries are given by
\begin{eqnarray*}
    D_M^Q(r,s)&=&\min\left\{D_{\searrow}^Q(r,s),D_{\nwarrow}^Q(r,s)\right\},\\
    D_O^Q(r,s)&=&\min\left\{D_{\swarrow}^Q(r,s),D_{\nearrow}^Q(r,s)\right\}.
\end{eqnarray*}
To simplify the notation we replace the subscript $Q$ by $A$ when we consider an ASM matrix $A$ instead of the corresponding \blue{discrete} quasi-copula $Q(A)$ in the computation of volumes $V_Q(R)$. Then we write $D_{\searrow}^A$, etc. In such a case, we refer to the defect matrices of an ASM, meaning the defect matrices of the corresponding \blue{discrete} quasi-copula $Q(A)\in\Q_n'$.

Using the defect matrices we obtain six transformation on \blue{(discrete)} quasi-copulas. They are:
\begin{eqnarray*}
    Q_{\searrow}&=&Q-D^Q_{\searrow},\\
    Q_{\swarrow}&=&Q+D^Q_{\swarrow},\\
    Q_{\nwarrow}&=&Q-D^Q_{\nwarrow},\\
    Q_{\nearrow}&=&Q+D^Q_{\nearrow},\\
    Q_{M}&=&Q-D^Q_{M},\\
    Q_{O}&=&Q+D^Q_{O}.
\end{eqnarray*}
Dibala et al \cite[Thm. 4.3]{DS-PMK} show that all six transformations applied to a \blue{(discrete)} quasi-copula yield \blue{(discrete)} quasi-copulas. 

\blue{
\begin{example}
    Consider the ASM $A$ and the discrete quasi-copula $Q=Q(A)$ given in Equation~\eqref{3x3 proper ASM}. The corresponding directional defect matrices are
$$D_{\searrow}^{Q}=\begin{pmatrix}
       -1 & 0 & 0 \\
       0  & 0 & 0 \\
       0  & 0 & 0
\end{pmatrix},\ 
    D_{\swarrow}^{Q}=\begin{pmatrix}
       0  & -1& 0 \\
       0  & 0 & 0 \\
       0  & 0 & 0
\end{pmatrix},\ 
    D_{\nwarrow}^{Q}=\begin{pmatrix}
       0  & 0 & 0 \\
       0  & -1& 0 \\
       0  & 0 & 0
\end{pmatrix}, 
     \ \mathrm{and}\ 
    D_{\nearrow}^{Q}=\begin{pmatrix}
        0 & 0 & 0 \\
        -1 & 0 & 0 \\
        0 & 0 & 0
\end{pmatrix}.
$$
The main and the opposite defect matrices are
$$
    D_M^{Q}=\begin{pmatrix}
       -1  & 0 & 0 \\
       0  & -1& 0 \\
       0  & 0 & 0
\end{pmatrix}, 
     \ \mathrm{and}\ 
    D_O^{Q}=\begin{pmatrix}
        0 & -1 & 0 \\
        -1 & 0 & 0 \\
        0 & 0 & 0
\end{pmatrix}.
$$
\end{example}
The six transformations applied to $Q$ give discrete quasi-copulas that are in fact discrete copulas:
$$
 Q_{\searrow}=\begin{pmatrix}
       1  & 1 & 1 \\
       1  & 1 & 2 \\
       1  & 2 & 3
\end{pmatrix}, \ 
    Q_{\swarrow}=\begin{pmatrix}
       0  & 0 & 1 \\
       1  & 1 & 2 \\
       1  & 2 & 3
\end{pmatrix},\ 
    Q_{\nwarrow}=\begin{pmatrix}
       0  & 1 & 1 \\
       1  & 2 & 2 \\
       1  & 2 & 3
\end{pmatrix}
$$
$$    Q_{\nearrow}=\begin{pmatrix}
       0  & 1 & 1 \\
       0  & 1 & 2 \\
       1  & 2 & 3
\end{pmatrix},\ 
    Q_{M}=\begin{pmatrix}
       1  & 1 & 1 \\
       1  & 2 & 2 \\
       1  & 2 & 3
\end{pmatrix},   \ \mathrm{and}\ 
    Q_{O}=\begin{pmatrix}
       0  & 0 & 1 \\
       0  & 1 & 2 \\
       1  & 2 & 3
\end{pmatrix}.
$$
Note that $Q_M=M_3$ and $Q_O=W_3$. It is interesting to observe that $Q$ together with its six defect transformations form the set of all discrete quasi-copulas of minimal range on $L_3^2$ (compare with \cite[Fig. 2]{BS}).
}

\subsection{Discrete imprecise copulas}
In this subsection, we introduce imprecise copulas and their desirable properties. \blue{We define when an imprecise copula is coherent and when it {avoids} sure loss}. We start by recalling the definition of an imprecise copula, \blue{first in the general setting and then in the discrete setting}. \blue{To streamline the notation, we here use $P$ and $Q$ for the lower and upper bound of an imprecise copula instead of $\underline{Q}$ and $\overline{Q}$ of the introduction. }

\blue{\begin{definition}
    A pair $(P,Q)$ of functions $P,Q:[0,1]^2\to [0,1]$ is called an \emph{imprecise copula} if $P$ and $Q$ satisfy the following conditions:
    \begin{enumerate}
        \item[(ICG1)]  Both $P$ and $Q$ satisfy condition (Q1).
        \item[(ICG2)] For each rectangle $R=[a,b]\times [c,d]$ we have:
        \begin{eqnarray*}
            Q(a,c)+P(b,d)-P(a,d)-P(b,c)&\geq & 0,\\
            P(a,c)+Q(b,d)-P(a,d)-P(b,c)&\geq & 0,\\
            Q(a,c)+Q(b,d)-Q(a,d)-P(b,c)&\geq & 0,\\
            Q(a,c)+Q(b,d)-P(a,d)-Q(b,c)&\geq & 0.
        \end{eqnarray*}
    \end{enumerate}
\end{definition}
}

\begin{definition}
    A pair $(P,Q)$ of functions $P,Q:L_n^2\to \blue{[0,n]}$ is called a \emph{\blue{discrete} imprecise copula} if $P$ and $Q$ satisfy the following conditions:
    \begin{enumerate}
        \item[(IC1)]  Both $P$ and $Q$ satisfy condition (Q1').
        \item[(IC2)] For each rectangle $R=[i,j]\times [k,l]\in\R$ we have:
        \begin{eqnarray*}
            Q(i,k)+P(j,l)-P(i,l)-P(j,k)&\geq & 0,\\
            P(i,k)+Q(j,l)-P(i,l)-P(j,k)&\geq & 0,\\
            Q(i,k)+Q(j,l)-Q(i,l)-P(j,k)&\geq & 0,\\
            Q(i,k)+Q(j,l)-P(i,l)-Q(j,k)&\geq & 0.
        \end{eqnarray*}
    \end{enumerate}
\end{definition}
It follows from a result of Montes et al. \cite[Prop. 2]{MMPV} that given a \blue{(discrete)} imprecise copula then both $P$ and $Q$ are quasi-copulas and $P(r,s)\leq Q(r,s)$ for all $(r,s)$, i.e., $P\leq Q$ in the \blue{point}-wise order. We also assume for any imprecise copula $(P,Q)$ that both $P$ and $Q$ are proper quasi-copulas.
Dibala et al \cite[Thm. 5.2]{DS-PMK} show that a pair $(P,Q)$ of quasi-copulas is an imprecise copula if and only if $P_M\le Q$ and $P\le Q_O$. There, the authors also observed that $(P,P_M)$ and $(Q_O,Q)$ are imprecise copulas as well. Repeating the operations on the first of the latter two imprecise copulas, one obtains $P\leq (P_M)_O\leq P_M\leq Q$, so that $((P_M)_O,P_M)$ is an imprecise copula inside the original imprecise copula. Iterating the process further, we get a sequence of embedded quasi-copulas $(P_k,Q_k)$ given by $P_0=P, Q_0=Q$ and $P_k=(Q_{k})_O, Q_k=(P_{k-1})_M$. The limiting pair $(\overline{P},\overline{Q})$ is also an imprecise copula which satisfies the equalities $(\overline{P})_M=\overline{Q}$ and $(\overline{Q})_O=\overline{P}$. See \cite[Prop. 10]{FinalS} for the results on the iteration process.

\begin{definition}
\label{def:imco}
    A \blue{(discrete)} imprecise copula $(P,Q)$ is \emph{self-dual} if $P_M=Q$ and $Q_O=P$. Alternatively, we say that a pair of \blue{(discrete)} quasi-copulas $(P,Q)$ is \emph{a dual pair} if $(P,Q)$ is a self-dual \blue{(discrete) imprecise copula}. Furthermore, if $P,Q\in\Q_n'$ and $A=A(P)$, $B=A(Q)$ then we say that a pair of ASM $(A,B)$ is \emph{a dual pair} if $(P,Q)$ is a self-dual \blue{discrete} imprecise copula, i.e., if $(P,Q)$ is a dual pair of \blue{discrete} quasi-copulas.  
\end{definition}

It follows from Definition~\ref{def:imco} that an imprecise copula $(P,Q)$ is self-dual if and only if $D_M^P=D_O^Q$. 

\medskip 

An example of a discrete imprecise copula $(P,Q)$, where both $P$ and $Q$ are quasi-copulas of minimal range of size $7$, is given by Omladi\v{c} and Stopar \cite[Ex. 11]{FinalS}. The same example gives a self-dual discrete imprecise copula which is then extended to a dual pair of proper full-domain quasi-copulas in \cite[Sect. 5]{K-BKOS}. In Section~\ref{isolated pairs}, we provide a generalization of this example to finer grids of size $n\ge 8$.

\begin{definition}\label{def_no sure loss}
    A discrete imprecise copula \blue{$(P,Q)$}  \blue{\emph{defined on $L_n^2$ avoids sure loss}} if there exists a discrete copula $C$ defined on $L_n^2$ such that $P\leq C\leq Q$. 
\end{definition}

\begin{definition}\label{def_coherence}
    Suppose that a \blue{discrete} imprecise copula $(P,Q)$  \blue{\emph{avoids sure loss}}, i.e., the set $\C(P,Q)$ of all copulas such that $P\leq C\leq Q$ is not empty. Then $(P,Q)$ is \emph{coherent} if $P(r,s)=\inf_{C\in\C(P,Q)}\{C(r,s)\}$ and $Q(r,s)=\sup_{C\in\C(P,Q)}\{C(r,s)\}$ for all $r,s\in L_n$. 
\end{definition}
Omladi\v{c} and Stopar \cite[\S 5]{FinalS} show that not every (discrete) imprecise copula \blue{avoids} sure loss nor it is coherent. Moreover, they provide a characterization of discrete imprecise copulas that \blue{avoid} sure loss \cite[Prop. 16]{FinalS}.

%The next section provides a general method to construct dual pairs of quasi-copulas.

\section{Patchwork construction}
\label{sec:3}
The first contribution of this paper is a general construction of dual pairs of quasi-copulas using patchwork techniques \cite{K-BKOS}. We
note that the results presented in this section apply to both discrete and full-domain quasi-copulas.
\blue{We first recall the definition of a patchwork construction for quasi-copulas on a single rectangle $R=[a,b]\times[c,d]\subset [0,1]^2$ as introduced in \cite[Thm. 1]{K-BKOS}. 
Since we are later interested in imprecise copulas, here we use a copula function as the `background' function. The patchwork construction given in \cite[Thm. 1]{K-BKOS} is more general, as it can be applied when the `background' function is a quasi-copula. 
However, in such a case, the analysis of defects is not straightforward. Therefore, we do not include this scenario here.

We consider a (discrete) copula $C$ (as the `background' function) that is defined at the four vertices of $R$. Next, we are given two pairs of univariate functions $F,F':[a,b]\to [0,1]$ and $G,G':[c,d]\to [0,1]$ that coincide with the background function $C$ at the the points on the edges of $R$ at which $C$ is defined, and that satisfy the boundary conditions, which are:
\begin{enumerate}
    \item[(i)] The values of these functions are uniquely determined at the vertices of $R$, i.e.,  $\alpha:=F(a) =G(c)$, $\beta :=F(b) =G'(c)$, $\gamma:=F'(b) =G'(d)$, and $\delta:=F'(a) =G(d)$.
\item[(ii)] These functions are increasing and $1$-Lipschitz.
\item[(iii)] For all $x \in [a, b]$ and $y\in[c, d]$ we have $0 \le F'(x) - F(x) \le d - c$ and $0 \le G'(y) - G(y)\le b - a$. 
\item[(iv)] The differences $F'(x)-F(x)$ and $G'(x)-G(x)$ are both monotone on their respective domains.
\item[(v)] The values at the end points are compatible with the $C$-volume: 
$$V_C(R)=(\gamma - \beta) - (\delta - \alpha) = (\gamma - \delta) - (\beta - \alpha),$$ 
where the four differences in the brackets are 'signed' lengths of the edges of $R$ given by the four boundary functions.
\end{enumerate}
Now, given a (discrete) quasi-copula $Q_A$, we define the patch for $Q$ on $R$ by letting
\begin{equation}\label{the patch}
Q(x, y) = F(x) + G(y) - \alpha + V_C(R)\cdot \overline{Q}_A (x, y),
\end{equation}
for any $(x, y) \in R$, where
$$\overline{Q}_A(x, y) = Q_A\left(\frac{F'(x) - F(x) - (\delta - \alpha)}{V_C(R)}, \frac{G'(y) - G(y) - (\beta - \alpha)}{V_C(R)}\right).
$$
\cite[Thm. 1]{K-BKOS} proves that the above construction gives a quasi-copula.

\blue{Suppose now that the unit square $I^2=[0,1]^2$ is partitioned into $m \in \mathbb{N}$ rectangles $R_k$, $k=1,\ldots,m$, that intersect only along their edges. 
Further, assume that we have a (discrete) copula $C$ that is defined at all of the vertices of rectangles $R_k$ and that we are given a (discrete) quasi-copula $Q_k$ for each $k=1,\ldots,m$, with $V_C(R_k)>0$. Further, we are given univariate boundary functions $F_k$, $F_k'$, $G_k$, $G_k'$ for each rectangle $R_k$ such that they satisfy the boundary conditions (i)-(v) and coincide along a common edge of any pair of rectangles. 
The patchwork quasi-copula $Q$ on $I^2$ (or $I_n^2$ in the discrete case) is then defined by Equation~\eqref{the patch} for each $R_k$ with $V_C(R_k)>0$. 
Note that if $V_C(R_l)=0$ for some $l$ then the copula $C$ is constant on such a rectangle $R_l$. 
We observe that the boundary conditions (i) and (ii) imply that then also the corresponding boundary functions are constant on the edges of $R_l$.  For such rectangles we take the constant patch, i.e., we define $Q$ to be constant on $R_l$ and its value equal to the value of $C$ on $R_l$.}

To further illustrate the idea behind this patchwork approach, we now present two examples of quasi-copulas (continuous and discrete) constructed through this technique.}

\textcolor{black}{
\begin{example}
\label{Ex:patch}
We partition the unit square in the four rectangles $R_1=[0,1/2]\times[0,1/2], R_2=[1/2,1]\times [0,1/2], R_3=[1/2,1]\times[1/2,1], R_4=[0,1/2]\times[1/2,1]$. 
As the `background function', we consider a copula $C$ such that it spreads the probability mass as follows: $C(R_1)=1/6$, $C(R_2)=1/3$, $C(R_3)=1/6$, and $C(R_4)=1/3$. 
Note that we can derive a discrete copula $DC$ associated with this partition of $C$ which corresponds to the following bistochastic matrix $B_C$:
\[B_C=\begin{pmatrix}1/3 & 2/3 \\ 2/3 & 1/3\end{pmatrix}.
\] 
The patchwork method described above allows for constructing proper quasi-copulas that preserve the probability mass assigned to each rectangle $R_i$, with $i=1,2,3,4$, by the background function, i.e., the copula $C$. We can choose any proper quasi-copula to redistribute the mass in each sub-rectangle of the partition. Figure~\ref{fig:expatch} shows two possible constructions of patchworks based on the specific setting of this example. 
The picture to the left of Figure~\ref{fig:expatch} represents a proper discrete quasi-copula obtained by using the following discrete quasi-copulas for the redistribution in each rectangle, i. e., $A_i$ on $R_i$:
\[ A_1 = A_3 = \begin{pmatrix}0 & 1 & 0 \\ 1 & -1 & 1 \\ 0 & 1 & 0 \end{pmatrix},\ A_2 = \begin{pmatrix}1 & 0 & 0 \\ 0 & 1 & 0 \\ 0 & 0 & 1 \end{pmatrix},\ \text{and}\ A_4 = \begin{pmatrix}0 & 0 & 1 \\ 0 & 1 & 0 \\ 1 & 0 & 0. \end{pmatrix}\]
The picture to the right of Figure~\ref{fig:expatch} shows the support of a possible patchwork construction that results in a proper full-domain quasi-copula. There, the black lines indicate the positive mass, while the red dashed lines indicate the negative mass.
The thickness of mass along the lines in $R_2$ and $R_3$ is twice the thickness of mass along the lines in the other two squares. A boundary function along any common edge of two rectangles is linear. It is determined by the value $Q(\frac12,\frac12)=\frac23$. \qed
\end{example}
}

\begin{figure}[h!]
\begin{tabular}{cc}
\begin{tikzpicture}[scale=0.4]
\fill[black!01] (1,1) rectangle (12,12);
\fill[red!12] (6.5,6.5) rectangle (12,12);
\fill[lincolngreen!15] (1,1) rectangle (6.5,6.5);
\fill[goldenpoppy!25] (1,6.5) rectangle (6.5,12);
\fill[blue!07] (6.5,1) rectangle (12,6.5);
\draw[black!50] (1,1) rectangle (12,12);
\draw (0.3,12.5)   node {{\scriptsize$(0,0)$}};
\draw (0.3,0.3)   node {{\scriptsize$(1,0)$}};
\draw (12.5,12.5)   node {{\scriptsize$(0,1)$}};
\draw (12.5,0.3)   node {{\scriptsize$(1,1)$}};
\draw[tue] (12.75,7.15)   node {{$\mathbf{R_2}$}};
\draw[safetyorange(blazeorange)] (0,7.15)   node {{$\mathbf{R_1}$}};
\draw[mediumjunglegreen] (0,5.75)   node {{$\mathbf{R_4}$}};
\draw[blue] (12.75,5.75)   node {{$\mathbf{R_3}$}};
\draw[black] (0,6.5) -- (13,6.5);
\draw[black] (6.5,1) -- (6.5,12);
\draw[dotted] (1,2.83) -- (12,2.83);
\draw[dotted] (1,4.66) -- (12,4.66);
\draw[dotted] (1,8.33) -- (12,8.33);
\draw[dotted] (1,10.16) -- (12,10.16);
\draw[dotted] (2.83,1) -- (2.83,12);
\draw[dotted] (4.66,1) -- (4.66,12);
\draw[dotted] (8.33,1) -- (8.33,12);
\draw[dotted] (10.16,1) -- (10.16,12);
\draw[dotted] (6.5,1) -- (6.5,12);
\draw[dotted] (1,6.5) -- (12,6.5);
\draw (1.915,1.915)   node {{\scriptsize$1/3$}};
\draw (1.915,3.745)   node {{\scriptsize$0$}};
\draw (1.915,5.575)   node {{\scriptsize$0$}};
\draw (1.915,7.405)   node {{\scriptsize$0$}};
\draw (1.915,9.235)   node {{\scriptsize$1/3$}};
\draw (1.915,11.065)   node {{\scriptsize$0$}};
\draw (3.745,1.915)   node {{\scriptsize$0$}};
\draw (3.745,3.745)   node {{\scriptsize$1/3$}};
\draw (3.745,5.575)   node {{\scriptsize$0$}};
\draw (3.745,7.405)   node {{\scriptsize$1/3$}};
\draw (3.745,9.235)   node {{\scriptsize$-1/3$}};
\draw (3.745,11.065)   node {{\scriptsize$1/3$}};
\draw (5.575,1.915)   node {{\scriptsize$0$}};
\draw (5.575,3.745)   node {{\scriptsize$0$}};
\draw (5.575,5.575)   node {{\scriptsize$1/3$}};
\draw (5.575,7.405)   node {{\scriptsize$0$}};
\draw (5.575,9.235)   node {{\scriptsize$1/3$}};
\draw (5.575,11.065)   node {{\scriptsize$0$}};
\draw (7.405,1.915)   node {{\scriptsize$0$}};
\draw (7.405,3.745)   node {{\scriptsize$2/3$}};
\draw (7.405,5.575)   node {{\scriptsize$0$}};
\draw (7.405,7.405)   node {{\scriptsize$0$}};
\draw (7.405,9.235)   node {{\scriptsize$0$}};
\draw (7.405,11.065)   node {{\scriptsize$2/3$}};
\draw (9.235,1.915)   node {{\scriptsize$2/3$}};
\draw (9.235,3.745)   node {{\scriptsize$-2/3$}};
\draw (9.235,5.575)   node {{\scriptsize$2/3$}};
\draw (9.235,7.405)   node {{\scriptsize$0$}};
\draw (9.235,9.235)   node {{\scriptsize$2/3$}};
\draw (9.235,11.065)   node {{\scriptsize$0$}};
\draw (11.065,1.915)   node {{\scriptsize$0$}};
\draw (11.065,3.745)   node {{\scriptsize$2/3$}};
\draw (11.065,5.575)   node {{\scriptsize$0$}};
\draw (11.065,7.405)   node {{\scriptsize$2/3$}};
\draw (11.065,9.235)   node {{\scriptsize$0$}};
\draw (11.065,11.065)   node {{\scriptsize$0$}};
\end{tikzpicture} &
\begin{tikzpicture}[scale=0.4]
\fill[black!01] (1,1) rectangle (12,12);
\fill[red!12] (6.5,6.5) rectangle (12,12);
\fill[lincolngreen!15] (1,1) rectangle (6.5,6.5);
\fill[goldenpoppy!25] (1,6.5) rectangle (6.5,12);
\fill[blue!07] (6.5,1) rectangle (12,6.5);
\draw[black!50] (1,1) rectangle (12,12);
\draw (0.3,12.5)   node {{\scriptsize$(0,0)$}};
\draw (0.3,0.3)   node {{\scriptsize$(1,0)$}};
\draw (12.5,12.5)   node {{\scriptsize$(0,1)$}};
\draw (12.5,0.3)   node {{\scriptsize$(1,1)$}};
\draw[black!30] (1,6.5) -- (12,6.5);
\draw[black!30] (6.5,1) -- (6.5,12);
\draw[dotted] (1,2.83) -- (12,2.83);
\draw[dotted] (1,4.66) -- (12,4.66);
\draw[dotted] (1,8.33) -- (12,8.33);
\draw[dotted] (1,10.16) -- (12,10.16);
\draw[dotted] (2.83,1) -- (2.83,12);
\draw[dotted] (4.66,1) -- (4.66,12);
\draw[dotted] (8.33,1) -- (8.33,12);
\draw[dotted] (10.16,1) -- (10.16,12);
\draw[dotted] (6.5,1) -- (6.5,12);
\draw[dotted] (1,6.5) -- (12,6.5);
\draw[black, thick] (1,8.33) -- (4.66,12);
\draw[black, thick] (2.83,6.5) -- (6.5,10.16);
\draw[red,dashed,thick] (2.83,10.16) -- (4.66,8.33);
\draw[black, thick] (8.33,1) -- (12,4.66);
\draw[black, thick] (6.5,2.83) -- (10.16,6.5);
\draw[red, dashed, thick] (10.16,2.83) -- (8.33,4.66);
\draw[black, thick] (1,1) -- (6.5,6.5);
\draw[black, thick] (12,6.5) -- (6.5,12);
\end{tikzpicture} 
\end{tabular}
\centering
\caption{\blue{Two examples of patchwork constructions with background function $C$ (a proper copula) and partition as specified in Example~\ref{Ex:patch}. Left: alternating bistochastic matrix corresponding to a patchwork discrete quasi-copula compatible with the background function $C$. Right: support of a full domain quasi-copula constructed through patchwork on the same partition The black lines indicate the positive mass, while the red dashed lines indicate the negative mass.}}
\label{fig:expatch}
\end{figure}

\blue{We now again assume that the unit square $I^2=[0,1]^2$ is partitioned}  into $m \in \mathbb{N}$ disjoint rectangles $R_k$, with $k=1,\ldots,m$. For $k=1,\ldots,m$, we consider $(Q_{A_k}, Q_{B_k})$ to be a pair of \blue{(discrete)} quasi-copulas. 
We further assume that the probability mass on the unit square is distributed according to a proper copula $C$. Equivalently, in each rectangle $R_k$ of the partition, the mass is non-negative and equals $V_C(R_k)$. 
%If $V_C(R_K)=0$ for some $k$ then we assume that $Q_{A_k}=Q_{B_k}=C_k$, where $C_k$ is a proper copula.
We now consider two \blue{(discrete)} quasi-copulas $Q_{A}$ and $Q_{B}$, constructed as patchworks of the \blue{(discrete)} quasi-copulas $Q_{A_k}$s and $Q_{B_k}$s, respectively. \blue{For the boundary functions we may take the values of the given copula $C$ on the edges of $R_k$, or we may use piecewise linear functions, or any other univariate functions that satisfy the boundary conditions (i)-(v) on each $R_k$ and coincide on the intersections of various $R_k$. In the patchwork construction, we use the same boundary functions on each rectangle $R_k$ for both $Q_A$ and $Q_B$.}

The underlying idea of a patchwork construction is to redistribute the mass $V_C(R_k)$ of each rectangle $R_k$ according to the \blue{(discrete)} quasi-copula $Q_{A_k}$. 
As proved in \cite[Thm. 5]{K-BKOS}, the function $Q_{A}$ obtained by gluing the components of each patchwork together is a proper quasi-copula. In a similar way, we can construct a \blue{(discrete)} quasi-copula $Q_{B}$ such that, for all $k=1,\ldots,m$, its mass on $R_k$, i.e., $V_C(R_k)$, is spread according to $Q_{B_k}$.

The following theorem states that a pair $(Q_{A},Q_{B})$ constructed as explained above inherits \blue{the property of being an imprecise copula and} the duality \blue{from} the pairs $(Q_{A_k},Q_{B_k})$.

\begin{theorem}
\label{Th:patch}
If $(Q_{A_k},Q_{B_k})$, with $k=1,\ldots,m %\in \mathbb{N}
$, are \blue{(discrete)} imprecise copulas then $(Q_{A},Q_{B})$ constructed through patchwork techniques is also a \blue{(discrete)}  imprecise copula. 

Furthermore, if $(Q_{A_k},Q_{B_k})$, with $k=1,\ldots,m %\in \mathbb{N}
$, are dual pairs of \blue{(discrete)} quasi-copulas for all $k$ such that $V_C(R_k)>0$, then the pair $(Q_{A},Q_{B})$ constructed through patchwork techniques is also a dual pair of \blue{(discrete)} quasi-copulas.
\end{theorem}

\blue{The proof is the same for the discrete and the full domain cases. In it, we do not use the adjective 'discrete' in parentheses for clarity. All rectangles are assumed to have their vertices in the domain of quasi-copulas considered.}

\begin{proof}
First, we notice that $(Q_A,Q_B)$ is an imprecise copula. Indeed, $Q_A$ and $Q_B$ are quasi-copulas, and $Q_A\leq Q_B$ since $Q_{A_k} \leq Q_{B_k}$ on each rectangle $R_k$. Conditions \blue{in Equation} (IC1) hold since they hold for $C$ and all $(Q_{A_k},Q_{B_k})$. If $R\in\R$ is a rectangle that lies entirely in one of $R_k$ then \blue{inequalities in Equation } (IC2) hold since they hold for ($Q_{A_k},Q_{B_k})$. Suppose that $R\in\R$ is such that it intersects nontrivially (i.e. not just along an edge of the patchwork) at least two of the rectangles $R_k$. We write $R=[i,j]\times [r,s]$. Consider the first of the inequalities in (IC2). Assume that the corner $(i,r)$ of $R$ lies in $R_k$ and that $R\cap R_k=[i,r']\times [j,s']$. Then we have
\begin{eqnarray*}
    Q_B(i,r)+Q_A(j,s)-Q_A(i,s)-Q_A(j,r)&=&r_k(Q_{B_k}(i,r')+Q_{A_k}(j,s')
    -Q_{A_k}(i,s')-Q_{A_k}(j,r'))\\
    &+&\sum_{l=1,l\neq k}^m r_l V_{Q_{A_l}}(R\cap R_l),
\end{eqnarray*}
where $r_l=V_C(R_l)$ and hence it is a nonnegative number for each $l$. Now, the first summand on the right-hand side of the above equality is nonnegative since $(Q_{A_k},Q_{B_k})$ is an imprecise copula and the volumes in the summation are nonnegative by Proposition \ref{prop:char-qc} since the rectangles $R\cap R_l$ have at least one boundary on the edges of the patchwork. Thus, all the summands on the right-hand side of the above equality are nonnegative, and the first inequality in (IC2) follows. Similar arguments show that the other three inequalities of (IC2) hold.

To prove that $(Q_A,Q_B)$ is a dual pair, we need to verify that $(Q_A)_M = Q_A - D_M^A = Q_B$ and $(Q_B)_O= Q_B+ D_O^B = Q_A$. We now show that the functions $D_M^A$ and $D_O^B$, i.e. the main defect of $Q_A$ and the opposite defect of $Q_B$, are straightforward combinations of the functions $D_M^{A_k}$ and, respectively, $D_O^{B_k}$, each on $R_k$ for $k=1, \ldots, m$. The only difference between the original $D_M^{A_k}$ and the corresponding sub-function in $D_M^{A}$ is a re-scaling factor $r_k$ which depends on $V_C(R_k)$ for each $R_k$. 
We notice that the rescaling factor $r_k$ is the same for $D_M^{A_k}$ and $D_O^{B_k}$. 
Indeed, we define the patchwork quasi-copulas $A$ and $B$ on the same partition of the unit square and the same distribution of the probability mass on $I^2$ according to a proper copula $C$. 

Without loss of generality, we here focus on the $D_M$ function, i.e. the main defect. The same arguments hold for the function $D_O$, i.e., the opposite defect.

We recall that for any $(r,s)\in [0,1]^2$, the value of the main defect function $D_M(r,s)$ is only affected by the rectangles on the main diagonal, i.e., with $(r,s)$ as a vertex, which have negative mass. 
We now show that we can restrict our attention to rectangles that lie entirely in one of the rectangles of the partition, i.e., they can only be like $\tilde{R}_b$ in Figure~\ref{fig:ex1}. 

\begin{figure}
\begin{tabular}{cc}
\begin{tikzpicture}[scale=0.4]
\fill[black!01] (1,1) rectangle (12,12);
\fill[red!15] (8,5.5) rectangle (12,12);
\fill[goldenpoppy!30] (1,5.5) rectangle (8,12);
\fill[lincolngreen!15] (1,1) rectangle (3.75,5.5);
\fill[blue!05] (3.75,1) rectangle (12,5.5);
\draw[black!50] (1,1) rectangle (12,12);
\draw (0.3,12.5)   node {{\scriptsize$(0,0)$}};
\draw (0.3,0.3)   node {{\scriptsize$(1,0)$}};
\draw (12.5,12.5)   node {{\scriptsize$(0,1)$}};
\draw (12.5,0.3)   node {{\scriptsize$(1,1)$}};
\draw[tue] (11,11.15)   node {{$\mathbf{R_2}$}};
\draw[safetyorange(blazeorange)] (1.75,11.15)   node {{$\mathbf{R_1}$}};
\draw[mediumjunglegreen] (1.75,1.75)   node {{$\mathbf{R_4}$}};
\draw[blue] (11,1.75)   node {{$\mathbf{R_3}$}};
\draw[black!20] (8,5.5) -- (12,5.5);
\draw[black!20] (3.75,1) -- (3.75,5.5);
\draw[dotted] (3.75,1) -- (3.75,12);
\draw[dotted] (1,8) -- (12,8);
\draw[black,pattern=north east lines, pattern color=black!50] (2,8) rectangle (3.75,10);
\draw[black,pattern=north east lines, pattern color=black!50] (3.75,3) rectangle (10,8);
\draw[black] (3,9)   node {{$\mathbf{\tilde{R}_b}$}};
\draw[black] (7,5.5)   node {{$\mathbf{\tilde{R}_a}$}};
\end{tikzpicture} &
\begin{tikzpicture}[scale=0.4]
\fill[black!01] (1,1) rectangle (12,12);
\fill[red!15] (8,5.5) rectangle (12,12);
\fill[goldenpoppy!30] (1,5.5) rectangle (8,12);
\fill[lincolngreen!15] (1,1) rectangle (3.75,5.5);
\fill[blue!05] (3.75,1) rectangle (12,5.5);
\draw[black!50] (1,1) rectangle (12,12);
\draw (0.3,12.5)   node {{\scriptsize$(0,0)$}};
\draw (0.3,0.3)   node {{\scriptsize$(1,0)$}};
\draw (12.5,12.5)   node {{\scriptsize$(0,1)$}};
\draw (12.5,0.3)   node {{\scriptsize$(1,1)$}};
\draw[tue] (11,11.15)   node {{$\mathbf{R_2}$}};
\draw[safetyorange(blazeorange)] (1.75,11.15)   node {{$\mathbf{R_1}$}};
\draw[mediumjunglegreen] (1.75,1.75)   node {{$\mathbf{R_4}$}};
\draw[blue] (11,1.75)   node {{$\mathbf{R_3}$}};
\draw[black!20] (8,5.5) -- (12,5.5);
\draw[black!20] (3.75,1) -- (3.75,5.5);
\draw[dotted] (3.75,1) -- (3.75,12);
\draw[dotted] (1,8) -- (12,8);
\draw[safetyorange(blazeorange), thick,pattern=north east lines, pattern color=red!70] (2,8) rectangle (3.75,10);
\draw[red, thick,pattern=north east lines, pattern color=red!70] (8,5.5) rectangle (10,8);
\draw[safetyorange(blazeorange), thick,pattern=north east lines, pattern color=safetyorange(blazeorange)!70] (3.75,8) rectangle (8,5.5);
\draw[blue, thick,pattern=north east lines, pattern color=blue!60] (10,5.5) rectangle (3.75,3);
 \draw[black, very thick] (3.65,7.9) rectangle (3.85,8.1);
\draw[black!20] (8,5.5) -- (12,5.5);
\draw[black!20] (3.75,1) -- (3.75,5.5);
\draw[dotted] (3.75,1) -- (3.75,12);
\draw[dotted] (1,8) -- (12,8);
\draw[black] (5.5,7)   node {{$\mathbf{\tilde{R}_{a1}}$}};
\draw[black] (9,7)   node {{$\mathbf{\tilde{R}_{a2}}$}};
\draw[black] (7.5,4.25)   node {{$\mathbf{\tilde{R}_{a3}}$}};
\draw[black] (3,9)   node {{$\mathbf{\tilde{R}_b}$}};
\end{tikzpicture}
\end{tabular}
\centering
\caption{Left: A graphical representation of a unit square partition with $m=4$. The patterned areas $\tilde{R}_a$ and $\tilde{R}_b$ are possible rectangles that contributes to $\R_{\nearrow}(r,s)$ and $\R_{\swarrow}(r,s)$. Right: the decomposition of the rectangles $\tilde{R}_{a}$ and $\tilde{R}_{b}$ according to the partition determined by $R_1,R_2,R_3,R_4$.}
\label{fig:ex1}
\end{figure}

Let $\tilde{R}$ be a rectangle such that $\tilde{R} \cap R_k \neq \emptyset$ for at least two indexes in $\{1,\ldots,m\}$ (see, for example, $\tilde{R}_a$ in Figure~\ref{fig:ex1}).
Then, the volume of $\tilde{R}$ can be decomposed as follows:
\[V_A(\tilde{R})= \sum_{k=1}^{m} V_A(\tilde{R}_k),\]
where $\tilde{R}_k= \tilde{R} \cap R_k$ for $k=1,\ldots m$. By construction, $(Q_A)_{|_{R_k}}=r_k Q_{A_k}$, where $r_k\geq 0$ is a rescaling factor that only depends on $V_C(R_k)$. As a consequence, $V(\tilde{R}_i)$ can only be negative if $r_k>0$ and $V_{Q_{A_i}}(\tilde{R}_i)$ is negative. However, $V_{{Q_{A_i}}}(\tilde{R}_i)$ must be non-negative because ${Q_{A_i}}$ is a quasi-copulas and, as such, it has non-negative mass on any rectangle with at least one edge on the border by Proposition \ref{prop:char-qc}.
Thus, the only rectangles that can contribute to a nonzero value of the function $D_M$ are those for which it exists one and only one $\tilde{k}$ such that they lie entirely in $R_{\tilde{k}}$ with $V_C(R_{\tilde{k}})>0$ and they do not touch its border.
These are the same rectangles that bring a negative value in the $D_M^{A_i}$ function. As a consequence, the function $D_M$ is a patchwork of the $D_M^{A_i}$ functions up to rescaling factors $r_k$. The same holds for the function $D_O$ which can be obtained by gluing all the $D_O^{B_i}$ functions re-weighted according to the same constants $r_k$. If $({Q_{A_k}},{Q_{B_k}})$ is a dual pair for every $k=1, \ldots,m$ such that $r_k>0$ then $D_M^{A_k}=D_O^{B_k}$ and, as a consequence, $D_M^A = D_O^B$. Hence, $(Q_A,Q_B)$ is also a dual pair. 
\end{proof}

Theorem~\ref{Th:patch} allows for combining dual pairs of quasi-copulas to obtain new dual pairs. In the remainder of the paper, we complement this result by giving construction techniques of dual pairs of quasi-copulas and study their properties. 

\section{Dense alternating sign matrices and dual pairs}
\label{sec:4}
We now move on to the second contribution of this paper and study the connection between dual pairs of \blue{discrete} quasi-copulas and a special class of ASM. In particular, we here show how to construct sequences of dual pairs of \blue{discrete} quasi-copulas in correspondence with a subset of ASM that are called \emph{dense} \cite{BS}.

\subsection{Dense alternating sign matrices}

\medskip

\begin{definition}
An ASM is called \emph{dense} if it has no zero entries between any two nonzero entries in either a column or in a row. 
\end{definition}

%\begin{example}
First we consider a special class of dense matrices that were studied by Brualdi and Schroeder \cite{BS}. Following them, we denote by $F_n^k$, $k=1,2,\ldots,n$, $n\ge 1$, a dense matrix that has all entries in four stripes $(k,1),(k-1,2)$,$\ldots, (1,k)$, $(1,k),(2,k+1),\ldots, (n-k+1,n)$, $(n-k+1,n),(n-k+2,n-1),\ldots,(n,n-k+1)$ and $(n,n-k+1),(n-1,n-k),\ldots, (1,k)$ equal to $1$. \blue{All entries between these stripes are nonzero by the denseness condition. Since the number of entries between the stripes is always odd, it is always possible to obtain an ASM.} In particular, $F_n^1$ is equal to the identity matrix $I_n$, and $F_n^n$ is the antidiagonal matrix \blue{$H_n$. These two matrices correspond to the Fr\'echet upper and lower bound, respectively. All other matrices $F_n^k$ are proper ASM.} For instance, when $n=4$ or $n=5$ we have
$$F_4^1=I_4=\begin{pmatrix}
    1 & 0 & 0 & 0 \\
    0 & 1 & 0 & 0\\
    0 & 0 & 1 & 0\\
    0 & 0 & 0 & 1
\end{pmatrix}, 
F_4^2=\begin{pmatrix}
    0 & 1 & 0 & 0 \\
    1 & -1 & 1 & 0\\
    0 & 1 & -1 & 1\\
    0 & 0 & 1 & 0
\end{pmatrix}, F_4^4=
\begin{pmatrix}
    0 & 0 & 0 & 1 \\
    0 & 0 & 1 & 0\\
    0 & 1 & 0 & 0\\
    1 & 0 & 0 & 0
\end{pmatrix}, $$
$$F_5^2=\begin{pmatrix}
    0 & 1 & 0 & 0 & 0 \\
    1 & -1 & 1 & 0 & 0\\
    0 & 1 & -1 & 1 & 0\\
    0 & 0 & 1 & -1 &  1\\
    0 & 0 & 0 & 1 & 0
\end{pmatrix}, 
F_5^3=\begin{pmatrix}
    0 & 0 & 1 & 0 & 0 \\
    0 & 1 & -1 & 1 & 0\\
    1 & -1 & 1 & -1 & 1\\
    0 & 1 & -1 & 1 & 0 \\
    0 & 0 & 1 & 0 & 0
\end{pmatrix}.$$
The quasi-copulas $Q(F_n^k)\in\Q_n'$ for the above ASM are
$$Q(F_4^1)=\begin{pmatrix}
    1 & 1 & 1 & 1 \\
    1 & 2 & 2 & 2\\
    1 & 2 & 3 & 3\\
    1 & 2 & 3 & 4
\end{pmatrix}=M_4,\  
Q(F_4^2)=\begin{pmatrix}
    0 & 1 & 1 & 1 \\
    1 & 1 & 2 & 2\\
    1 & 2 & 2 & 3\\
    1 & 2 & 3 & 4
\end{pmatrix},\  
Q(F_4^4)=
\begin{pmatrix}
    0 & 0 & 0 & 1 \\
    0 & 0 & 1 & 2\\
    0 & 1 & 2 & 3\\
    1 & 2 & 3 & 4
\end{pmatrix}=W_4, $$
$$Q(F_5^2)=
\begin{pmatrix}
    0 & 1 & 1 & 1 & 1 \\
    1 & 1 & 2 & 2 & 2\\
    1 & 2 & 2 & 3 & 3\\
    1 & 2 & 3 & 3 & 4\\
    1 & 2 & 3 & 4 & 5
\end{pmatrix}, 
Q(F_5^3)=\begin{pmatrix}
    0 & 0 & 1 & 1 & 1 \\
    0 & 1 & 1 & 2 & 2\\
    1 & 1 & 2 & 2 & 3\\
    1 & 2 & 2 & 3 & 4\\
    1 & 2 & 3 & 4 & 5
\end{pmatrix}.$$

\medskip

\blue{Next we consider a} sum of two consecutive matrices $F_n^{k-1}$ and $F_n^k$.

\blue{\begin{proposition}\label{consec_F_n^k}
    The sum $F_n^{k-1}+F_n^k$ for $2\leq k\leq n$ is a nonnegative matrix. Furthermore, it is the sum of two permutation matrices.
\end{proposition}
}

\textcolor{black}{\begin{proof}
If $k<\frac{n}{2}$ then the negative entries of $F_n^{k-1}$ corresponds to the positive entries of the matrix $F_n^k$. Therefore, they cancel out if we consider their sum, resulting in the following:
\begin{equation}\label{example F_n^k}
F_n^{k-1}+F_n^k=\left(\begin{array}{ccccccccccc}
 0 & \cdots  & 0 & 1 & 1 & 0 & & \cdots & \cdots & & 0\\
 0 & \iddots & 1 & 0 & 0 & 1 & & \ddots & \ddots & & 0\\
 & \iddots & \iddots & & & & \ddots & \ddots &\ddots & \ddots & \\
 1 & 0 &  & \cdots & & \cdots & & 1 & 0 & \cdots & 0\\
 1 & 0 &  & \cdots & & \cdots & & 0 & 1 & \cdots & 0\\
 & \ddots & \ddots & & & & & \ddots & \ddots & \ddots & \\
 0 & \cdots & 1 & 0 & & \cdots & & \cdots & & 0 & 1\\
 0 & \cdots & 0 & 1 & & \cdots & & \cdots & & 0 & 1\\
 & \ddots & \ddots & \ddots & \ddots & & & & \iddots & \iddots & \\
 0 & & \ddots & \ddots & & 1 & 0 & 0 & 1 & \iddots & 0\\
 0 & & \cdots & \cdots & & 0 & 1 & 1 & 0 & \cdots & 0
\end{array}\right).
\end{equation}
A similar calculation shows that the observation holds also for $k\ge\frac{n}{2}$. 
Observe also that one can write the sum $F_n^{k-1}+F_n^k$ as a sum of two permutation matrices by replacing every second nonzero entry in the clockwise direction around the center of the matrix by a zero. 
\end{proof}
}

In particular, for the case when $n=6$ and $k=5$, we have 
$$F^4_6+F^5_6=\begin{pmatrix}
    0 & 0 & 0 & 1 & 1 & 0 \\
    0 & 0 & 1 & 0 & 0 & 1 \\
    0 & 1 & 0 & 0 & 0 & 1 \\
    1 & 0 & 0 & 0 & 1 & 0 \\
    1 & 0 & 0 & 1 & 0 & 0 \\
    0 & 1 & 1 & 0 & 0 & 0 
\end{pmatrix},
$$
\blue{and}%For the above case we \blue{also} have
$$F^4_6+F^5_6=\begin{pmatrix}
    0 & 0 & 0 & 1 & 0 & 0 \\
    0 & 0 & 0 & 0 & 0 & 1 \\
    0 & 1 & 0 & 0 & 0 & 0 \\
    0 & 0 & 0 & 0 & 1 & 0 \\
    1 & 0 & 0 & 0 & 0 & 0 \\
    0 & 0 & 1 & 0 & 0 & 0 
\end{pmatrix}+\begin{pmatrix}
    0 & 0 & 0 & 0 & 1 & 0 \\
    0 & 0 & 1 & 0 & 0 & 0 \\
    0 & 0 & 0 & 0 & 0 & 1 \\
    1 & 0 & 0 & 0 & 0 & 0 \\
    0 & 0 & 0 & 1 & 0 & 0 \\
    0 & 1 & 0 & 0 & 0 & 0 
\end{pmatrix}.$$
%\end{example}

Each ASM can be written as the sum of its positive and negative parts: $A=A^++A^-$. For instance,
$$F_5^{3+}=\begin{pmatrix}
    0 & 0 & 1 & 0 & 0 \\
    0 & 1 & 0 & 1 & 0\\
    1 & 0 & 1 & 0 & 1\\
    0 & 1 & 0 & 1 & 0 \\
    0 & 0 & 1 & 0 & 0
\end{pmatrix}\ \mathrm{and}\ F_5^{3-}=\begin{pmatrix}
    0 & 0 & 0 & 0 & 0 \\
    0 & 0 & -1 & 0 & 0\\
    0 & -1 & 0 & -1 & 0\\
    0 & 0 & -1 & 0 & 0 \\
    0 & 0 & 0 & 0 & 0
\end{pmatrix}.$$

\begin{definition}
    A dense ASM $A$ is called \emph{irreducible} if it is equal to the $1\times 1$ matrix $[1]$, or it is equal to $F_n^k$ for some $n\ge 3$ and $2\le k \le n-1$.
\end{definition}

\begin{theorem}\label{structure of dense ASM}
    Suppose that $A$ is a dense ASM of size $n$. Then there exist an integer $l$, irreducible dense ASM matrices $A_r$, $r=1,2,\ldots, l$, and a permutation $\sigma$ of $\{1,2,\ldots,l\}$ such that $A$ has a block form
    $$A=\begin{pmatrix}
 A_{11} & A_{12} & \cdots  & A_{1l} \\
 A_{21} & A_{22} & \cdots  & A_{2l} \\
 \vdots & \vdots & & \vdots\\
 A_{l1} & A_{l2} & \cdots  & A_{ll} 
\end{pmatrix},$$
where for each $r=1,2,\ldots,l$, we have $A_{r,\sigma(r)}=A_r$  and $A_{rs}=0$ for $s\neq \sigma(r)$. Here $A_r=F_{m_r}^{k_r}$ for some integers $m_r$ such that $\sum_{r=1}^l m_r=n$ and either $m_r=k_r=1$ or $m_r\ge 3$ and $2\le k_r\le m_r-1$.
\end{theorem}

\begin{proof}
    We denote by $b_i$, $i=1,2,\ldots, n$ the number of nonzero entries in the $i$-th row of $A$. Note that $b_1=1$. Consider first the case $b_2=1$ as well. Then $a_{1,j_1}$ for some $j_1$ is the only nonzero entry in the first row, and at the same time it is the only nonzero entry in the $j_1$-th column of $A$. If the latter was not the case, then we would have $a_{2,j_1}=-1$ since $A$ is dense, and hence $b_2>1$, in contradiction with our assumption. 
    
    Assume now that $b_2>1$ and let $h\ge 3$ be the smallest index such that $b_h=1$. This $h$ exists since $b_n=1$. Consider now rows with indices $1,2,\ldots,h$. Since $A$ is dense the sequence $b_1,b_2,\ldots, b_h$ is unimodal and of form 
    \begin{equation}\label{unimodal seq}
        1,3,5,\ldots, 2g+1\ldots, 2g+1,2g-1,\ldots, 3,1
    \end{equation} 
    for some $g\ge 1$. We notice that if the sequence was not unimodal of the given form, then there would be a zero entry in a column between two nonzero entries. We denote by $f$ the number of $2g+1$ occurring in the sequence. Then we have $h=2g+f$. Since $A$ is dense, it follows that the nonzero entries in the first $h$ rows are necessarily in consecutive columns, say columns with indices from $j_1$ to $j_2-1$. Observe that $j_2-j_1=2g+f=h$ since the sequence \eqref{unimodal seq} requires $2g+1$ consecutive columns with nonzero entries for the first $g+1$ rows and then one additional column for each repetition of $2g+1$ as the sign of entries in the $A$ alternate. Together, there are $2g+1+f-1=h$ columns. Thus the block 
    $$\begin{pmatrix}
    a_{1j_1} & a_{1,j_1+1} & \cdots  & a_{1,j_2-1} \\
    a_{2j_1} & a_{2,j_1+1} & \cdots  & a_{2,j_2-1} \\
    \vdots & \vdots & & \vdots\\
    a_{hj_1} & a_{h,j_1+1} & \cdots  & a_{h,j_2-1} 
    \end{pmatrix}$$  
    of $A$ is of the form $F_h^k$, where either $k=g$ or $k=g+f-1$. Since $A$ is dense, it follows that the entries of $A$ in columns from $j_1$ to $j_2-1$ and in rows $h+1$ to $n$ are equal to zero. If $n=h$ we are done. Otherwise, it follows that $b_{h+1}=1$, and we consider the subsequence $b_{h+1},\ldots, b_{h'}$, where $h'=h+1$ if $b_{h'2}=1$ or $h'$ is the smallest index larger than $h+1$ such that $b_{h'}=1$ if $b_{h'}=1$. We repeat the above arguments to obtain the next block of the form $F_h^k$ until all rows are used. Let $l$ denote the number of subsequnces we considered in the process. Then $A$ has $l\times l$ block form, where in each row and column exactly one of the blocks is nonzero. We denote by $\sigma$ the permutation of $\{1,2,\ldots,l\}$ such that $A_{r,\sigma(r)}$ are the nonzero blocks of $A$.
\end{proof}

\subsection{Defect matrices of a dense ASM}

For a proper dense ASM $A$ the defect matrices of $A$, or the corresponding proper quasi-copula $Q(A)$, have a special form that will help us to determine the dual pairs. For two matrices $A$ and $B$ we denote by $A\wedge B$ the entry-wise minimum of $A$ and $B$.

\begin{proposition}\label{V_Q=-1}
    Suppose that $A$ is a dense ASM of size $n$ and $Q=Q(A)\in\Q_n'$ is the corresponding \blue{discrete} quasi-copula. Then we have $V_Q(R)\ge -1$ for any rectangle $R=[i,j]\times [k,l]$ in $\R$ and $V_Q(R)=-1$ if and only if both $j-i$ and $l-k$ are odd numbers and the values $a_{i,k}$, $a_{j,l}$, $a_{j,k}$ and $a_{i,l}$ of $A$ at the corners of $R$ are all equal to $-1$. 
\end{proposition}

\begin{proof}
    First, consider a rectangle with all four corner values $a_{i,k}$, $a_{j,l}$, $a_{j,k}$ and $a_{i,l}$ non-zero. If either of $j-i$ and $l-k$ is even, then the number of $1$s and the number of $-1$s in summation \eqref{V_Q(A)} are equal, and so $V_Q(R)=0.$ If both $j-i$ and $l-k$ are odd, then $V_Q(R)$ is equal to $1$ or to $-1$. The latter is the case if and only if the number of $-1$s is larger than the number of $1$s in the summation \eqref{V_Q(A)}, and this is possible if and only if all corner values are equal to $-1$, 
   
    Next, we assume that at least one of the corner values $a_{i,k}$, $a_{j,l}$, $a_{j,k}$ and $a_{i,l}$ is equal to $0$. We want to prove that, in this case, $V_Q(R)\ge 0$. We proceed by induction on $j-i$. If $j-i=0$ then at least one of $a_{i,k}$ and $a_{i,l}$ is equal to zero, and so $V_Q(R)=\sum_{s=0}^{l-k} a_{i,k+s}\in\{0,1\}$ by the alternating property of the entries in an ASM. If $j-i\ge 1$ then decompose $R$ into a union of two smaller rectangles. If $a_{i,k}=0$ or $a_{i,l}=0$, then we consider $R_1=[i,i]\times [k,l]$ and $R_2=[i+1,j]\times [k,l]$, which results in $V_Q(R_1)\in\{0,1\}$ as in the case $j-i=0$. If $R_2$ has a zero corner value, then its volume $V_Q(R_2)$ is non-negative by induction. Otherwise, all the corner values of $R_2$ are non-zero, implying that $a_{i+1,k}=1$ or $a_{i+1,l}=1$, since $a_{i,k}=0$ or $a_{i,l}=0$. By the arguments of the first part of the proof, it follows that $V_Q(R_2)\ge 0$. Hence $V_Q(R)=V_Q(R_1)+V_Q(R_2)\ge 0$. 
    The case where both $a_{i,k}$ and $a_{i,l}$ are nonzero and at least one of the other two corner values is zero is done in a similar way.
\end{proof}

\begin{corollary}\label{defects for dense ASM}
    Suppose that $A$ is a proper dense ASM of size $n$ and $Q=Q(A)\in\Q_n'$ is the corresponding proper \blue{discrete} quasi-copula. Then, the entries of all the defect matrices are equal to $0$ or $-1$. An entry is equal to $-1$ if and only if the corner entry of $A$ in the defect direction (one of the corner entries in the two defect directions for $D_M^Q$ and $D_O^Q$) is equal to $-1$. In particular,
    \begin{enumerate}
        \item[(a)] $D_{\searrow}^Q(r,s)=-1$ if and only if $a_{r+1,s+1}=-1$,
        \item[(b)] $D_{\swarrow}^Q(r,s)=-1$ if and only if $a_{r+1,s}=-1$,
        \item[(c)] $D_{\nwarrow}^Q(r,s)=-1$ if and only if $a_{r,s}=-1$, 
        \item[(d)] $D_{\nearrow}^Q(r,s)=-1$ if and only if $a_{r,s+1}=-1$, 
        \item[(e)] $D_M^Q(r,s)=-1$ if and only if $\min\{ a_{r+1,s+1},a_{r,s}\}=-1$,
        \item[(f)] $D_O^Q(r,s)=-1$ if and only if $\min\{a_{r+1,s},a_{r,s+1} \}=-1$.
    \end{enumerate}
\end{corollary}

\begin{lemma}\label{defects F}
    Suppose that $A$ is a proper dense ASM of size $n$. Then the defect matrices of $A$, or the corresponding proper \blue{discrete} quasi-copula $Q(A)$, are:
    \begin{eqnarray*}
    D^A_{\searrow}&=&J_nA^-J_n^T,\\
    D^A_{\swarrow}&=&J_nA^-,\\
    D^A_{\nwarrow}&=&A^-,\\
    D^A_{\nearrow}&=&A^-J_n^T,\\
    D^A_{M}&=&A^-\wedge J_nA^-J_n^T,\\
    D^A_{O}&=&J_nA^-\wedge A^-J_n^T.
    \end{eqnarray*}
\end{lemma}

\begin{proof}
    The result follows directly by applying Corollary \ref{defects for dense ASM}. 
\end{proof}

\subsection{A maximal chain of dual pairs}
\blue{We now derive the main results of Section~\ref{sec:4}. We show that for each $n$ the set of dense ASM is invariant under the main and the opposite transformations. 
The discrete quasi-copulas corresponding to irreducible dense ASM together with $M$ and $W$ provide a maximal chain for the two operations on discrete quasi-copulas that start and end with $W$ or $M$. In this way, we also obtain a maximal chain of dual pairs. 
In addition, we prove that any dual pair in the chain is coherent and thus avoids sure loss.}
\begin{theorem}\label{main and opposite of F}
    Suppose $n\ge 3$. Then 
    $$Q(F_n^k)_M=Q(F_n^{k-1})\  \textit{and}\ Q(F_n^k)_O=Q(F_n^{k+1})$$ 
    for $2\le k\le n-1$.
\end{theorem}

\begin{proof}
    We now show that $D_M^{F_n^k}=D_O^{F_n^{k-1}}$ for $2\leq k\leq n$. %By Lemma \ref{defects F} this holds if and only if 
    %$$\left(F^{k}_n\right)^-\wedge J_n\left(F_n^{k}\right)^-J_n^T=J_n\left(F_n^{k-1}\right)^-\wedge \left(F_n^{k-1}\right)^-J_n^T.$$
    The corollary \ref{defects for dense ASM}\textit{(e)} ensures that $D_M^{F_n^k}(r,s)=-1$ if and only if $r=k+t-u$, $s=t+u-1$ for $t=1,2,\ldots,n-k$ and $u=1,2,\ldots,k$.  By Corollary \ref{defects for dense ASM}\textit{(f)}, we have that $D_O^{F_n^k}(r,s)=-1$ if and only if $r=k+t-u-1$, $s=t+u-1$ for $t=1,2,\ldots,n-k+1$ and $u=1,2,\ldots,k-1$. Therefore, combining these results, we obtain the desired equality $D_M^{F_n^k}=D_O^{F_n^{k-1}}$.
    %    \textit{To be written.}
\end{proof}

\begin{corollary}\label{maximal chain}
    Suppose that $n\ge 4$. Then $(F^{k}_n,F_n^{k-1})$ is a dual pair of ASM for $3\le k\le n-1$. So, \blue{discrete} imprecise copulas $\left(Q(F_n^{n-1}),Q(F_n^{n-2})\right)$, $\left(Q(F_n^{n-2}),Q(F_n^{n-3})\right)$, \ldots, $\left(Q(F_n^{3}),Q(F_n^{2})\right)$ form a maximal chain of self-dual \blue{discrete} imprecise copulas, i.e., a chain $\{Q_r\}_{r=1}^l$, $l\ge 2$, of proper \blue{discrete} quasi-copulas such that $(Q_r)_O=Q_{r-1}$ for $r=2,3,\ldots, l$ and $(Q_r)_M=Q_{r+1}$ for $r=1,2,\ldots,l-1$. The chain cannot be extended further since $Q(F_n^{n-1})_O=W$ and $Q(F_n^{2})_M=M$.  
\end{corollary}

We denote by $P_n^{(1)}$ the diagonal matrix with diagonal entries alternating between $1$ and $0$. We write $P_n^{(2)}=I_n-P_n^{(1)}$. For instance, if $n=5$ then we have
$$P_5^{(1)}=\begin{pmatrix}
    1 & 0 & 0 & 0 & 0 \\
    0 & 0 & 0 & 0 & 0\\
    0 & 0 & 1 & 0 & 0\\
    0 & 0 & 0 & 0 & 0 \\
    0 & 0 & 0 & 0 & 1
\end{pmatrix}\ \mathrm{and}\ P_5^{(2)}=\begin{pmatrix}
    0 & 0 & 0 & 0 & 0 \\
    0 & 1 & 0 & 0 & 0\\
    0 & 0 & 0 & 0 & 0\\
    0 & 0 & 0 & 1 & 0 \\
    0 & 0 & 0 & 0 & 0
\end{pmatrix}.$$

\begin{theorem}\label{Fnk coherent}
     Suppose that $n\ge 4$. Then each discrete imprecise copula $\left(Q(F^{k}_n),Q(F_n^{k-1})\right)$ for $3\le k\le n-1$  \blue{ is coherent. In particular, it avoids sure loss.}
\end{theorem}

\begin{proof}
    We first focus on showing that $\left(Q(F^{k}_n),Q(F_n^{k-1})\right)$ \blue{avoids} sure loss, that is, there exists a proper  \blue{discrete} copula $C$ in between the two \blue{discrete} quasi-copulas $Q(F^{k}_n)$ and $Q(F_n^{k-1})$. From the \blue{decomposition in \eqref{example F_n^k}} it follows that the midpoint $A=\frac12(F_n^{k}+F_n^{k-1})$ is a matrix with all nonnegative entries. Therefore, the corresponding \blue{discrete} quasi-copula $C=Q(A)$ is in fact a \blue{discrete} copula and we have $Q(F^{k}_n)\le C\le Q(F_n^{k-1})$ by construction. So $\left(Q(F^{k}_n),Q(F_n^{k-1})\right)$ \blue{avoids} sure loss. 

    We write $E_1=P_n^{(1)}D_M^{F_n^k}$ and $E_2=P_n^{(2)}D_M^{F_n^k}$. We want to show that $C_1=Q(F_n^k)-E_1$ and $C_2=Q(F_n^k)-E_2$ are \blue{discrete} copulas such that $Q(F^{k}_n)\le C_i\le Q(F_n^{k-1})$ for $i=1,2$. In fact, they are  \blue{discrete}copulas corresponding to the two permutation matrices obtained by writing $F_n^{k}+F_n^{k-1}$ as the sum of the two permutation matrices \blue{by Proposition \ref{consec_F_n^k}.} It is also possible to observe that the ASM matrix corresponding to $C_i$, $i=1,2$, is nonnegative by computing matrices $A(E_i)=(K_n)^{-1} E_i K_n^{-1}$. When deriving $A(E_i)$, we observe that each $2\times 2$ block $\left(\begin{array}{rr}
        -1 & 0 \\
         0 & 0
    \end{array}\right)$ in $E_i$ is replaced by  $\left(\begin{array}{rr}
        -1 & 1 \\
         1 & -1
    \end{array}\right)$ in $A(E_i)$. Then $A(C_i)=F_n^{k}-A(E_i)$ is a nonnegative matrix.    
    Furthermore, it is easy to observe that $C_1$ and $C_2$ are constructed so that for each entry we have 
    $$Q(F^{k}_n)(r,s)=\min\{C_1(r,s),C_2(r,s)\}\ \textrm{ and }\ Q(F_n^{k-1})(r,s)=\max\{C_1(r,s),C_2(r,s)\}.$$ 
    Hence, discrete imprecise copula $\left(Q(F^{k}_n), Q(F_n^{k-1})\right)$ is coherent.
%    
%    To complete the proof we have to show that $C_1$ and $C_2$ are copulas: [to complete] 
%    
\end{proof}

\medskip

We note that Brualdi and Schroeder \cite[Thm. 4]{BS} proved that the interval $(F_n^{k},F_n^{k-1})$ in the lattice $\A_n$ is isomorphic as a lattice to the lattice of all subsets of a set of cardinality $k(n-k)$.

\subsection{Characterization of dense dual pairs}
\blue{In this last part of Section~\ref{sec:4}, we combine the results proved previously in the section and derive a characterization of dense dual pairs in terms of their nonzero block structure.}
\begin{theorem}
    A dense ASM $A$ is such that $(Q(A),Q(A)_M)$ is a \blue{discrete} imprecise copula if and only if all its nonzero blocks are of the form $F_{m_r}^{k_r}$ for some integers $m_r$ such that $\sum_{r=1}^l m_r=n$ and either $m_r=k_r=1$ or $m_r\ge 4$ and $3\le k_r\le m_r-1$. Furthermore, we have $m_r\ge 4$ for at least one $r$. 

    A dense ASM $A$ is such that $(Q(A)_O,Q(A))$ is a \blue{discrete} imprecise copula if and only if all its nonzero blocks are of the form $F_{m_r}^{k_r}$ for some integers $m_r$ such that $\sum_{r=1}^l m_r=n$ and either $m_r=k_r=1$ or $m_r\ge 4$ and $2\le k_r\le m_r-2$. Additionally, we have $m_r\ge 4$ for at least one $r$.
\end{theorem}

\begin{proof}
    By Theorem \ref{structure of dense ASM} $A$ is of a special block form with the non-zero block being an irreducible dense ASM. The nonzero blocks are of the form $F_{n_u}^{k_u}$, $u=1,2,\ldots, l$. Note that the corresponding \blue{discrete} quasi-copula is a patchwork of  \blue{discrete} quasi-copulas \blue{that correspond to irreducible dense ASM} and $0$ blocks. Then $(Q,Q_M)$ is a self-dual \blue{discrete} imprecise copula if and only if each of its blocks $Q_{rs}$ is either zero or satisfies the condition that $(Q_{rs},\left(Q_{rs}\right)_M)$ is a self-dual \blue{discrete} imprecise copula. The first part then follows by the fact that $F_n^k$ induces a self-dual \blue{discrete} imprecise copula $(Q(F_n^k),Q(F_n^k)_M)$ if and only if $n\ge 4$ and $3\le k\le n-1$ (see Corollary \ref{maximal chain}). For the second part, we use that $F_n^k$ induces a  \blue{discrete} imprecise copula $(Q(F_n^k)_O,Q(F_n^k))$ if and only if $n\ge 4$ and $2\le k\le n-2$.
\end{proof}

\section{An example of non-dense dual pairs for square grids of size $n\geq 7$}
\label{isolated pairs}
We now discuss a construction of dual pairs which is not covered by the results of Section~\ref{sec:4}. 
The dual-pair construction proposed in this section originates from \cite[Example 11]{FinalS}. \blue{We begin with the case $n=7$ for illustration to ease the understanding of more technical general case, and then move to the general case of an arbitrary grid size $n > 7$.}

\subsection{Non-dense pair for a fixed grid size $n=7$}

Let $Q({A_7})$ and $Q({B_7})$ be the proper discrete quasi-copulas defined as follows:
\[Q({A_7})=\begin{pmatrix}
                        0 & 0  & 0 & 0 & 0 & 1 & 1 \\
                        0  & 0 & 0 & 0 & 1 & 2 & 2\\
                        0 & 0 & 0 & 1 & 2 & 2 & 3\\
                        0 & 0 & 1 & 2 & 2 & 3 & 4\\                        
                        0 & 1 & 2 & 2 & 3 & 4 & 5 \\
                        1 & 2 & 2 & 3 & 4 & 5 & 6\\
                        1 & 2 & 3 & 4 & 5 & 6 & 7 \\
\end{pmatrix} \quad \text{and} \quad
Q({B_7})=\begin{pmatrix}
                        0 & 0 & 0 & 0 & 1 & 1 & 1 \\
                        0 & 0 & 0 & 1 & 2 & 2 & 2 \\
                        0 & 0 & 1 & 2 & 2 & 3 & 3 \\
                        0 & 1 & 2 & 2 & 3 & 4 & 4\\
                        1 & 2 & 2 & 3 & 4 & 4 & 5 \\
                        1 & 2 & 3 & 4 & 4 & 5 & 6\\
                        1 & 2 & 3 & 4 & 5 & 6 & 7 \\
\end{pmatrix},
\]
whose corresponding alternating sign matrices $A_7$ and $B_7$ are as given below:
\[A_7=\begin{pmatrix}
                        0 & 0  & 0 & 0 & 0 & 1 & 0 \\
                        0  & 0 & 0 & 0 & 1 & 0 & 0\\
                        0 & 0 & 0 & 1 & 0 & -1 & 1\\
                        0 & 0 & 1 & 0 & -1 & 1 & 0\\
                        0 & 1 & 0 & -1 & 1 & 0 & 0 \\
                        1 & 0 & -1 & 1 & 0 & 0 & 0\\
                        0 & 0 & 1 & 0 & 0 & 0 & 0 \\                        
\end{pmatrix} \quad \text{and} \quad
B_7=\begin{pmatrix}
                        0 & 0  & 0 & 0 & 1 & 0 & 0 \\
                        0  & 0 & 0 & 1 & 0 & 0 & 0\\
                        0 & 0 & 1 & 0 & -1 & 1 & 0\\
                        0 & 1 & 0 & -1 & 1 & 0 & 0\\
                        1 & 0 & -1 & 1 & 0 & -1 & 1 \\
                        0 & 0 & 1 & 0 & -1 & 1 & 0\\
                        0 & 0 & 0 & 0 & 1 & 0 & 0 \\                        
\end{pmatrix}.
\]

We notice that the matrices $A_7$ and $B_7$ display an interesting pattern in their distribution of the positive, zero, and negative mass. 
Namely, $A_7$ has zero mass on the (second) diagonal, which makes a clear cut between the positive and the negative mass. Analogously, $B_7$ displays an alternating double pattern of positive, zero, and negative mass.
In this section, we prove that such a distribution of positive, zero, and negative mass always leads to a dual pair of discrete quasi-copulas. The zero separation makes this construction different than the case of dense ASM presented in Section~\ref{sec:4}, and it is somewhat complementary to it.
Before considering the general case of $(n\times n)$ ASM, we take a closer look at the defect matrices of the $7 \times 7$ example. 
By computing the defects, one can notice that non-zero entries of the defect matrices $D_M^{A_7}$, and $D_O^{B_7}$ are determined by the alternating distribution of the positive, zero, and negative mass in the original quasi-copulas. 
Namely, the entries of the matrix $D_M^{A_7}$ can be derived as follows:
\[D_M^{A_7} = 
\left\{
\begin{aligned}
D_M^{A_7}(7-i-1,i)& = D_{\searrow}^{A_7}(7-i-1,i) = -1 \quad  & \text{for} \quad &i=1,2,3,4,5 \\
D_M^{A_7}(7-i,i)& =D_{\searrow}^{A_7}(7-i,i) =  -1 \quad & \text{for} \quad &i=2,3,4,5 \\
D_M^{A_7}(7-i+1,i+1)& = D_{\nwarrow}^{A_7}(7-i+1,i+1) = -1 \quad & \text{for} \quad & i=2,3,4,5 \\
D_M^{A_7}(7-i+1,i+2)& = D_{\nwarrow}^{A_7}(7-i+1,i+2) = -1 \quad & \text{for} \quad & i=2,3,4 \\
D_M^{A_7}(i,j) &= 0 \quad & & \text{otherwise}
\end{aligned}
\right.\]
Similarly, for the defect matrix $D_O^{B_7}$, we have that:
\[D_O^{B_7} = 
\left\{
\begin{aligned}
D_O^{B_7}(7-i-1,i)& = D_{\nearrow}^{B_7}(7-i-1,i) = -1 \quad  & \text{for} \quad &i=1,2,3 \\
D_O^{B_7}(7-i-1,i)& = D_{\swarrow}^{B_7}(7-i-1,i) = -1 \quad  & \text{for} \quad &i=4,5 \\
D_O^{B_7}(7-i,i)& =D_{\nearrow}^{B_7}(7-i,i) =  -1 \quad & \text{for} \quad &i=2,3 \\
D_O^{B_7}(7-i,i)& =D_{\swarrow}^{B_7}(7-i,i) =  -1 \quad & \text{for} \quad &i=4,5 \\
D_O^{B_7}(7-i+1,i+1)& = D_{\nearrow}^{B_7}(7-i+1,i+1) = -1 \quad & \text{for} \quad & i=2,3 \\
D_O^{B_7}(7-i+1,i+1)& = D_{\swarrow}^{B_7}(7-i+1,i+1) = -1 \quad & \text{for} \quad & i=4,5 \\
D_O^{B_7}(7-i+1,i+2)& = D_{\nearrow}^{B_7}(7-i+1,i+2) = -1 \quad & \text{for} \quad & i=2,3 \\
D_O^{B_7}(7-i+1,i+2)& = D_{\swarrow}^{B_7}(7-i+1,i+2) = -1 \quad & \text{for} \quad & i=4 \\
D_O^{B_7}(i,j) &= 0 \quad & & \text{otherwise}
\end{aligned}
\right.\]
Fig.~\ref{fig:dMA} and Fig.~\ref{fig:dob} show a depiction of the types of rectangles that contribute to the negative entries of the defect matrices $D_M^{A_7}$ and $D_O^{B_7}$.
The defect matrices $D_M^{A_7}$ and $D_O^{B_7}$ are equal \cite[Example 11]{FinalS} and take the form
\[
D_M^{A_7}=D_O^{B_7}=\begin{pmatrix}
                        0 & 0  & 0 & 0 & -1 & 0 & 0 \\
                        0  & 0 & 0 & -1 & -1 & 0 & 0\\
                        0 & 0 & -1 & -1 & 0 & -1 & 0\\
                        0 & -1 & -1 & 0 & -1 & -1 & 0\\
                        -1 & -1 & 0 & -1 & -1 & 0 & 0 \\
                        0 & 0 & -1 & -1 & 0 & 0 & 0\\
                        0 & 0 & 0 & 0 & 0 & 0 & 0 \\
\end{pmatrix}.
\]

\begin{figure}
\centering
\begin{minipage}{.45\linewidth}
    \begin{tikzpicture}[scale=0.4]
\draw (1,1) rectangle (15,15);
\draw (15,1) -- (15,15);
\draw (1,1) -- (15,1);
\draw (0.5,0.3)   node {{\scriptsize $(1,0)$}};
\draw (16,15.5)   node {{\scriptsize$(0,1)$}};
\draw (15.5,0.3)   node {{\scriptsize$(1,1)$}};
\draw (0.3,15.5)   node {{\scriptsize$(0,0)$}};

\draw[dotted] (3,1) -- (3,15);
\draw[dotted] (5,1) -- (5,15);
\draw[dotted] (7,1) -- (7,15);
\draw[dotted] (9,1) -- (9,15);
\draw[dotted] (11,1) -- (11,15);
\draw[dotted] (13,1) -- (13,15);
\draw[dotted] (1,3) -- (15,3);
\draw[dotted] (1,5) -- (15,5);
\draw[dotted] (1,7) -- (15,7);
\draw[dotted] (1,9) -- (15,9);
\draw[dotted] (1,11) -- (15,11);
\draw[dotted] (1,13) -- (15,13);

\draw[ultra thick,dotted, umass,fill=umass!10, fill opacity=0.5] (7,9) rectangle (11,5);
\draw[thick,lincolngreen,fill=lincolngreen!10, fill opacity=0.5] (7.1,7.1) rectangle (11.05,9.05);
\draw[red,pattern=north east lines, pattern color=red] (9.05,7.05) rectangle (11.05,8.95);
\draw (2,4)   node {\textcolor{black}{\textbf{1}}};
\draw (4,6)   node {\textcolor{black}{\textbf{1}}};
\draw (6,8)   node {\textcolor{black}{\textbf{1}}};
\draw (8,10)   node {\textcolor{black}{\textbf{1}}};
\draw (10,12)   node {\textcolor{black}{\textbf{1}}};
\draw (12,14)   node {\textcolor{black}{\textbf{1}}};

 \draw (6,4)   node {\textcolor{black}{\textbf{-1}}};
 \draw (8,6)   node {\textcolor{black}{\textbf{-1}}};
 \draw (10,8)   node {\textcolor{black}{\textbf{-1}}};
 \draw (12,10)   node {\textcolor{black}{\textbf{-1}}};

\draw (6,2)   node {\textcolor{black}{\textbf{1}}};
\draw (8,4)   node {\textcolor{black}{\textbf{1}}};
\draw (10,6)   node {\textcolor{black}{\textbf{1}}};
\draw (12,8)   node {\textcolor{black}{\textbf{1}}};
\draw (14,10)   node {\textcolor{black}{\textbf{1}}};
\end{tikzpicture}
    \caption{Representation of the non-zero mass of $A_7$ and depiction of the types of rectangles that contribute to the non-zero entries of the defect matrix $D_M^{A_7}$.}
    \label{fig:dMA}
\end{minipage}
\hspace{.05\linewidth}
\begin{minipage}{.45\linewidth}
\begin{tikzpicture}[scale=0.4]
\draw (1,1) rectangle (15,15);
\draw (15,1) -- (15,15);
\draw (1,1) -- (15,1);
\draw (0.5,0.3)   node {{\scriptsize $(1,0)$}};
\draw (16,15.5)   node {{\scriptsize$(0,1)$}};
\draw (15.5,0.3)   node {{\scriptsize$(1,1)$}};
\draw (0.3,15.5)   node {{\scriptsize$(0,0)$}};

\draw[dotted] (3,1) -- (3,15);
\draw[dotted] (5,1) -- (5,15);
\draw[dotted] (7,1) -- (7,15);
\draw[dotted] (9,1) -- (9,15);
\draw[dotted] (11,1) -- (11,15);
\draw[dotted] (13,1) -- (13,15);
\draw[dotted] (1,3) -- (15,3);
\draw[dotted] (1,5) -- (15,5);
\draw[dotted] (1,7) -- (15,7);
\draw[dotted] (1,9) -- (15,9);
\draw[dotted] (1,11) -- (15,11);
\draw[dotted] (1,13) -- (15,13);

 \draw[thick,lincolngreen,fill=lincolngreen!10, fill opacity=0.5] (3.05,5.05) rectangle (7,7);
 \draw[red,pattern=north east lines, pattern color=red] (5.05,5.05) rectangle (6.95,6.95);

 \draw[ultra thick,dotted, umass,fill=umass!10, fill opacity=0.5] (9,5) rectangle (13,7);
 \draw[safetyorange(blazeorange),pattern=north west lines, pattern color=safetyorange(blazeorange)] (11.05,5.05) rectangle (12.95,6.95);

\draw (2,6)   node {\textcolor{black}{\textbf{1}}};
\draw (4,8)   node {\textcolor{black}{\textbf{1}}};
\draw (6,10)   node {\textcolor{black}{\textbf{1}}};
\draw (8,12)   node {\textcolor{black}{\textbf{1}}};
\draw (10,14)   node {\textcolor{black}{\textbf{1}}};

 \draw (6,6)   node {\textcolor{black}{\textbf{-1}}};
 \draw (8,8)   node {\textcolor{black}{\textbf{-1}}};
 \draw (10,10)   node {\textcolor{black}{\textbf{-1}}};
 
\draw (6,4)   node {\textcolor{black}{\textbf{1}}};
\draw (8,6)   node {\textcolor{black}{\textbf{1}}};
\draw (10,8)   node {\textcolor{black}{\textbf{1}}};
\draw (12,10)   node {\textcolor{black}{\textbf{1}}};

 \draw (10,4)   node {\textcolor{black}{\textbf{-1}}};
 \draw (12,6)   node {\textcolor{black}{\textbf{-1}}};

\draw (10,2)   node {\textcolor{black}{\textbf{1}}};
\draw (12,4)   node {\textcolor{black}{\textbf{1}}};
\draw (14,6)   node {\textcolor{black}{\textbf{1}}};
\end{tikzpicture}
    \caption{Representation of the non-zero mass of $B_7$ and depiction of the types of rectangles that contribute to the non-zero entries of the defect matrix $D_O^{B_7}$.}
    \label{fig:dob}
\end{minipage}
\end{figure}

%We now compute the difference between the special quasi-copulas $Q(A_7)$ and $Q(B_7)$: 
%\[Q(A_7) - Q(B_7)= (d^7_{i,j})_{i,j=1}^7=
%\left\{
%\begin{aligned}
%d^7_{7-i-1,i}& =  -1 \quad  & \text{for} \quad &i=1,2,3,4,5 \\
%d^7_{7-i,i}& =  -1 \quad & \text{for} \quad &i=2,3,4,5 \\
%d^7_{7-i+1,i+1}& =  -1 \quad & \text{for} \quad & i=2,3,4,5 \\
%d^7_{7-i+1,i+2}& =  -1 \quad & \text{for} \quad & i=2,3,4 \\
%d^7_{i,j} &= 0 \quad & & \text{otherwise}
%\end{aligned}
%\right.,
%\]
%Now, $Q(A_7) - Q(B_7) =  D_M^{A_7}$, which implies that $(Q(A_7))_M = Q(A_7) - D_M^{A_7} = Q(B_7)$. The converse $(Q(B_7))_O= Q(B_7)+ D_O^{B_7} = Q(A_7)$ follows directly from the equality of the defect matrices. Hence, the duality of pair $(Q(A_7), Q(B_7))$.

\subsection{Construction of non-dense pairs for $n\ge 7$}
\label{subsec:ispai_n}

The following result generalizes the example above to arbitrary matrix sizes $n\geq7$. 

\begin{proposition}
\label{pro:isolat_pairs_n}
Suppose $Q(A_n)=\left({q^{A_n}_{i,j}}\right)_{i,j = 1}^n$ and $Q(B_n)$ are two proper quasi-copulas defined as follows:
\[
Q(A_n) = %({q^{A_n}_{i,j}})_{i,j = 1}^n=
\left\{
\begin{aligned}
q^{A_n}_{i,j} &= 0 \quad    & \text{for} \quad &j=1,\ldots,n-2, \quad i\leq n-j-1 \\
q^{A_n}_{n,1} &= q^{A_n}_{1,n}=1 = q^{A_n}_{n-i,i} & \text{for} \quad & i=1,\ldots,n-1 \\
q^{A_n}_{n,2} &= q^{A_n}_{2,n} = 2 = q^{A_n}_{n-i+1,i} = q^{A_n}_{n-i+1,i+1} \quad & \text{for} \quad & i=2,\ldots, n-1\\
q^{A_n}_{n-j,i+1} &= n+j-i
%q^{A_n}_{n-j,i} + 1 
\quad & \text{for} \quad & j=0,\ldots,n-2,\\
& & \text{and} \quad & i=j+2,\ldots,\min\{n-1, j+2\} 
\end{aligned}
\right.
\]

and $Q(B_n) = Q(A_n) - D$, where $D= \left[d_{i,j}\right]_{i,j = 1}^n $ is given by:
\[
D= %(d_{i,j})_{i,j = 1}^n = 
\left\{
\begin{aligned}
d_{n-i-1,i}& =  -1 \quad  & \text{for} \quad &i=1,\ldots,n-2 \\
d_{n-i,i}& =  -1 \quad & \text{for} \quad &i=2,\ldots, n-2 \\
d_{n-i+1,i+1}& =  -1 \quad & \text{for} \quad & i=2,\ldots, n-2 \\
d_{n-i+1,i+2}& =  -1 \quad & \text{for} \quad & i=2,\ldots, n-3 \\
d_{i,j} &= 0 \quad & & \text{otherwise}
\end{aligned}
\right.
\]

Then $Q(A_n)$ and $Q(B_n)$ form a dual pair.
\begin{proof}
%\ep{We first write the expression for the quasi copulas $Q(A_n)$ and $Q(B_n)$. Then we derive the defect matrices and claim that they are equal. Then we show that ${A_n}_M = {A_n} - D_M^{A_n} = B_n$ and ${B_n}_O= B_n+ D_O^{B_n} = A_n$}
The discrete quasi-copulas $Q(A_n)$ and $Q(B_n)$ are natural extensions of the matrices $Q(A_7)$ and $Q(B_7)$ to arbitrary grid domains of size $n$. As such, they correspond to the following alternating sign matrices $A_n = \left(a_{i,j}\right)_{i,j = 1}^n$ and $B_n = \left(b_{i,j}\right)_{i,j = 1}^n$ defined as:
\[
A_n = \left\{
\begin{aligned}
a_{n-i+1,i-1}&= 1 \quad  & \text{for} \quad &i=2,\ldots,n\\
a_{n-i,i+1}& = -1 \quad & \text{for} \quad &i=2,\ldots, n-2 \\
a_{n-i+1,i+2}& = 1 \quad & \text{for} \quad & i=1,\ldots,n-2 \\
a_{i,j} &= 0 \quad & & \text{otherwise}
\end{aligned}
\right. 
\quad
B_n = \left\{
\begin{aligned}
b_{n-i+1,i-2}& = 1 \quad  & \text{for} \quad &i=3,\ldots,n\\
b_{n-i+1,i}& = -1 \quad & \text{for} \quad &i=3,\ldots, n-2 \\
b_{n-i,i+1}& = 1 \quad & \text{for} \quad & i=2,\ldots,n-3 \\
b_{n-i+1,i+3}& = -1 \quad & \text{for} \quad & i=2,\ldots,n-4 \\
b_{n-i+1,i+4}& = 1 \quad & \text{for} \quad & i=1,\ldots,n-4 \\
b_{i,j} & = 0 \quad & & \text{otherwise}
\end{aligned}
\right.
\]
Since the mass distribution of $A_n$ and $B_n$ follows the same pattern as for the case $n=7$, the defect matrices $D_M^{A_n}$ and $D_O^{B_n}$ can be easily derived in a similar way. Therefore, $D_M^{A_n}$ and $D_O^{B_n}$ are given by the following expressions:
\[D_M^{A_n} = 
\left\{
\begin{aligned}
D_M^{A_n}(n-i-1,i)& = D_{\searrow}^{A_n}(n-i-1,i) = -1 \quad  & \text{for} \quad &i=1,\ldots,n-2 \\
D_M^{A_n}(n-i,i)& =D_{\searrow}^{A_n}(n-i,i) =  -1 \quad & \text{for} \quad &i=2,\ldots,n-2\\
D_M^{A_n}(n-i+1,i+1)& = D_{\nwarrow}^{A_n}(n-i+1,i+1) = -1 \quad & \text{for} \quad & i=2,\ldots,n-2 \\
D_M^{A_n}(n-i+1,i+2)& = D_{\nwarrow}^{A_n}(n-i+1,i+2) = -1 \quad & \text{for} \quad & i=2,\ldots,n-3 \\
D_M^{A_n}(i,j) &= 0 \quad & & \text{otherwise}
\end{aligned}
\right.\]
In a similar way, for the defect matrix $D_O^{B_n}$, we have that:
\[D_O^{B_n} = 
\left\{
\begin{aligned}
D_O^{B_n}(n-i-1,i)& = %D_{\searrow}^{B_n}(n-i-1,i) = 
-1 \quad  & \text{for} \quad &i=1,\ldots,n-2  \\
D_O^{B_n}(n-i,i)& = %D_{\searrow}^{B_n}(n-i,i) =  
-1 \quad & \text{for} \quad &i=2,\ldots,n-2  \\
D_O^{B_n}(n-i+1,i+1)& = %D_{\searrow}^{B_n}(n-i+1,i+1) = 
-1 \quad & \text{for} \quad & i=2,\ldots,n-2  \\
D_O^{B_n}(n-i+1,i+2)& = %D_{\searrow}^{B_n}(n-i+1,i+2) = 
-1 \quad & \text{for} \quad & i=2,\ldots,n-3 \\
D_O^{B_n}(i,j) &= 0 \quad & & \text{otherwise}
\end{aligned}
\right.\]
We notice that the matrix $D$ equals the defect matrices $D_O^{B_n}$ and $D_M^{A_n}$. In particular, $Q(B_n) = Q(A_n) - D = Q(A_n) - D_M^{A_n}$ and $Q(A_n) = Q(B_n) + D_M^{A_n} = Q(B_n) + D_O^{B_n}$. Hence, $(Q(A_n),Q(B_n))$ is a dual pair.
\end{proof}
\end{proposition}

\subsection{\blue{The constructed pairs in section~\ref{subsec:ispai_n} are coherent and avoid sure loss}}

As for the dual pairs constructed in Section~\ref{sec:4}, a natural question arises whether or not the dual pairs presented in this section \blue{avoid} sure loss, i.e., whether there exists a proper copula $C$ such that $Q(A_n) \leq C \leq Q(B_n)$, and whether or not they are coherent.
The answer to these questions is affirmative, as shown in the example below.

\begin{example}\label{AB7_is_coherent}
    We first illustrate the reasoning for $n=7$. We consider again the dual pair $(Q(A_7),Q(B_7))$. One may easily check that the following three \blue{discrete} quasi-copulas $C_i$, $i=1,2,3$, are such that $Q(A_7)\le C_i\le Q(B_7)$:
    \[C_1=\begin{pmatrix}
                        0 & 0 & 0 & 0 & 1 & 1 & 1 \\
                        0 & 0 & 0 & 1 & 2 & 2 & 2\\
                        0 & 0 & 0 & 1 & 2 & 2 & 3\\
                        0 & 1 & 1 & 2 & 3 & 3 & 4\\                        
                        1 & 2 & 2 & 3 & 4 & 4 & 5 \\
                        1 & 2 & 2 & 3 & 4 & 5 & 6\\
                        1 & 2 & 3 & 4 & 5 & 6 & 7 \\
\end{pmatrix}, \quad
C_2=\begin{pmatrix}
                        0 & 0 & 0 & 0 & 1 & 1 & 1 \\
                        0 & 0 & 0 & 0 & 1 & 2 & 2\\
                        0 & 0 & 1 & 1 & 2 & 3 & 3\\
                        0 & 1 & 2 & 2 & 3 & 4 & 4\\                        
                        0 & 1 & 2 & 2 & 3 & 4 & 5 \\
                        1 & 2 & 3 & 3 & 4 & 5 & 6\\
                        1 & 2 & 3 & 4 & 5 & 6 & 7 \\
\end{pmatrix}, 
\]
\[
\text{and} \quad
C_3=\begin{pmatrix}
                        0 & 0 & 0 & 0 & 0 & 1 & 1 \\
                        0 & 0 & 0 & 1 & 1 & 2 & 2 \\
                        0 & 0 & 1 & 2 & 2 & 3 & 3 \\
                        0 & 0 & 1 & 2 & 2 & 3 & 4\\
                        1 & 1 & 2 & 3 & 3 & 4 & 5 \\
                        1 & 2 & 3 & 4 & 4 & 5 & 6\\
                        1 & 2 & 3 & 4 & 5 & 6 & 7 \\
\end{pmatrix}.
\]
The alternating sign matrices $P_i$, $i=1,2,3$, corresponding to each $C_i$ are:
\[P_1=\begin{pmatrix}
                        0 & 0 & 0 & 0 & 1 & 0 & 0 \\
                        0 & 0 & 0 & 1 & 0 & 0 & 0\\
                        0 & 0 & 0 & 0 & 0 & 0 & 1\\
                        0 & 1 & 0 & 0 & 0 & 0 & 0\\
                        1 & 0 & 0 & 0 & 0 & 0 & 0 \\
                        0 & 0 & 0 & 0 & 0 & 1 & 0\\
                        0 & 0 & 1 & 0 & 0 & 0 & 0 \\                        
\end{pmatrix}, \quad
P_2=\begin{pmatrix}
                        0 & 0 & 0 & 0 & 1 & 0 & 0 \\
                        0 & 0 & 0 & 0 & 0 & 1 & 0\\
                        0 & 0 & 1 & 0 & 0 & 0 & 0\\
                        0 & 1 & 0 & 0 & 0 & 0 & 0\\
                        0 & 0 & 0 & 0 & 0 & 0 & 1 \\
                        1 & 0 & 0 & 0 & 0 & 0 & 0\\
                        0 & 0 & 0 & 1 & 0 & 0 & 0 \\                        
\end{pmatrix}, 
\]
\[
\text{and} \quad
P_3=\begin{pmatrix}
                        0 & 0 & 0 & 0 & 0 & 1 & 0 \\
                        0 & 0 & 0 & 1 & 0 & 0 & 0\\
                        0 & 0 & 1 & 0 & 0 & 0 & 0\\
                        0 & 0 & 0 & 0 & 0 & 0 & 1\\
                        1 & 0 & 0 & 0 & 0 & 0 & 0 \\
                        0 & 1 & 0 & 0 & 0 & 0 & 0\\
                        0 & 0 & 0 & 0 & 1 & 0 & 0 \\                        
\end{pmatrix}
\]
Since the matrices $P_i$ do not have any negative entry, they are permutation matrices and, as a consequence, $C_1, C_2,$ and $C_3$ are proper discrete copulas that satisfy the condition $Q(A_7) \leq C_i \leq Q(B_7)$ for $i=1,2,3$. Hence, the discrete imprecise copula $(Q(A_7), Q(B_7))$ \blue{avoids} sure loss.

Now, we investigate whether the pair $(Q(A_7), Q(B_7))$ is also coherent. 
We start by focusing on the $16$ matrix indexes $(r,s)$ for which the entries in $Q(A_7)$ and $Q(B_7)$ are distinct. For such indexes, there is a $C_i$ such that its entry into position $(r,s)$ is equal to the entry in position $(r,s)$ of $Q(A_7)$, and a $C_j$ such that its entry in position $(r,s)$ is equal to the entry in position $(r,s)$ of $Q(B_7)$. 
This shows that $Q(A_7)$ is the lower bound of the set of all copulas $\C=\{C;\ Q(A_7)\leq C\leq Q(B_7)\}$ and that $Q(B_7)$ is the upper bound of $\C$. Therefore, the self-dual \blue{discrete} imprecise copula $(Q(A_7),Q(B_7))$ is coherent.

\end{example}

Before considering the general case $n\ge 7$, we introduce three matrices $E_i$ defined as the differences $E_i=C_i-Q(A_7)$: %We omit the last row and column of $E_i$ that are identically equal to $0$ and we denote the truncated matrices by $E_i'$. Hence, 
\blue{
\[E_1=\begin{pmatrix}
                        0 & 0 & 0 & 0 & 1 & 0 & 0 \\
                        0 & 0 & 0 & 1 & 1 & 0 & 0 \\
                        0 & 0 & 0 & 0 & 0 & 0 & 0 \\
                        0 & 1 & 0 & 0 & 1 & 0 & 0 \\
                        1 & 1 & 0 & 1 & 1 & 0 & 0 \\
                        0 & 0 & 0 & 0 & 0 & 0 & 0 \\ 
                        0 & 0 & 0 & 0 & 0 & 0 & 0 \\      
\end{pmatrix}, \quad
E_2=\begin{pmatrix}
                        0 & 0 & 0 & 0 & 1 & 0 & 0 \\
                        0 & 0 & 0 & 0 & 0 & 0 & 0 \\
                        0 & 0 & 1 & 0 & 0 & 1 & 0 \\
                        0 & 1 & 1 & 0 & 1 & 1 & 0 \\
                        0 & 0 & 0 & 0 & 0 & 0 & 0 \\
                        0 & 0 & 1 & 0 & 0 & 0 & 0 \\
                        0 & 0 & 0 & 0 & 0 & 0 & 0 \\      
\end{pmatrix}, 
\]
\[
\text{and} \quad
E_3=\begin{pmatrix}
                        0 & 0 & 0 & 0 & 0 & 0 & 0 \\
                        0 & 0 & 0 & 1 & 0 & 0 & 0 \\
                        0 & 0 & 1 & 1 & 0 & 1 & 0 \\
                        0 & 0 & 0 & 0 & 0 & 0 & 0 \\
                        1 & 0 & 0 & 1 & 0 & 0 & 0 \\
                        0 & 0 & 1 & 1 & 0 & 0 & 0 \\   
                        0 & 0 & 0 & 0 & 0 & 0 & 0 \\      
\end{pmatrix}.
\]
}
We observe that, for each $E_i$, the entries on the first anti-diagonal above the main anti-diagonal, and on the second anti-diagonal below the main anti-diagonal form (part of) the sequence $1,1,0,1,1,0$. Furthermore, the entry on the main anti-diagonal or on the third anti-diagonal below the main anti-diagonal is nonzero and equal to $1$ if and only if both the entry to the left and above are equal to $1$. The other entries are all $0$. 

We now extend Example~\ref{AB7_is_coherent} to arbitrary grid size $n \geq 7$ by generalizing the pattern in \blue{$E_i$}. Namely, we define the matrices \blue{$E_{n,i}=\left[e_{rs}^{(n,i)}\right]_{r,s=1}^{n}$}, $i=1,2,3$, as follows: 

\begin{description}
\item[] {\blue{$\mathbf{E_{n,1}}$:}} The sequence of entries $e_{1,n-2}^{(n,1)},e^{(n,1)}_{2,n-3},\ldots,e^{(n,1)}_{n-2,1},e^{(n,1)}_{3,n-1},e^{(n,1)}_{4,n-2},\ldots,e^{(n,1)}_{n-1,3}$ is periodic with a period of length three of the form $1,1,0$. The entry $e^{(n,1)}_{ij}$ with $i+j\in\{n,n+3\}$ is equal to $1$ if and only if $e^{(n,1)}_{i,j-1}=e^{(n,1)}_{i-1,j}=1$. All other entries are equal to $0$.
\item[] $\blue{\mathbf{E_{n,2}}}$: The sequence of entries $e_{1,n-2}^{(n,2)},e^{(n,2)}_{2,n-3},\ldots,e^{(n,2)}_{n-2,1},e^{(n,2)}_{3,n-1},e^{(n,2)}_{4,n-2},\ldots,e^{(n,2)}_{n-1,3}$ is periodic with a period of length three of the form $1,0,1$. The entry $e^{(n,2)}_{ij}$ with $i+j\in\{n,n+3\}$ is equal to $1$ if and only if $e^{(n,2)}_{i,j-1}=e^{(n,2)}_{i-1,j}=1$. All other entries are equal to $0$.
\item[]$\blue{\mathbf{E_{n,3}}}$: The sequence of entries $e_{1,n-2}^{(n,3)},e^{(n,3)}_{2,n-3},\ldots,e^{(n,3)}_{n-2,1},e^{(n,3)}_{3,n-1},e^{(n,3)}_{4,n-2},\ldots,e^{(n,3)}_{n-1,3}$ is periodic with a period of length three of the form $0,1,1$. The entry $e^{(n,3)}_{ij}$ with $i+j\in\{n,n+3\}$ is equal to $1$ if and only if $e^{(n,3)}_{i,j-1}=e^{(n,3)}_{i-1,j}=1$. All other entries are equal to $0$.
\end{description}
%We augment matrices $E_{n,i}'$ with a zero column to the right and a zero row at and 
We consider $C_{n,i}=Q(A_n)+E_{n,i}$. Now, $C_{n,i}$ are \blue{discrete} copulas such that $Q(A_n)\leq C_{n,i}\leq Q(B_n)$ for all $i$. Furthermore, for each $r,s\in L_n$, we have 
    $$Q(A_n)(r,s)=\min\{C_{n,1}(r,s),C_{n,2}(r,s),C_{n,3}(r,s)\}, \quad Q(B_n)(r,s)=\max\{C_{n,1}(r,s),C_{n,2}(r,s),C_{n,3}(r,s)\}.$$
Therefore, the dual pair $(A_n,B_n)$ is coherent.
%\ep{we still need the generalization to $n>7$}.

\section{Conclusions}
\label{sec:Conclusions}
In this paper, we study imprecise copulas on a discrete grid domain. We present construction techniques to obtain self-dual discrete imprecise copulas. We focus on the discrete imprecise copulas that correspond to the special class of dense Alternating Sign Matrices \cite{BS}, which are of general interest in discrete geometry. Additionally, we generalize Example~11 of \cite{FinalS} and obtain dual pairs that are non-dense. In all cases, we prove that the constructed pairs \blue{avoid} sure loss and are coherent.
We complement our constructions with a general result showing how to patch together self-dual discrete imprecise copulas to maintain duality in the discrete and continuous settings.

\smallskip
This work focuses on the duality of specific extreme points of the polytope of discrete quasi-copulas, i.e., those that correspond to Alternating Sign Matrices with determined entry patterns. 
In the future, we will investigate how properties propagate in more general discrete settings. \blue{In particular, we are in interested in dual pairs, coherent quasi-copulas, and imprecise copulas that avoid sure loss.}
Since all our constructions result in dual pairs that \blue{avoid} sure loss, a natural question would be whether or not it is possible to find a dual pair of extreme quasi-copulas that does not satisfy this property. 
Additionally, one could also study how the mass distribution of a dual-pair of quasi-copulas can be modified to preserve duality. The non-dense self-dual imprecise copulas presented in Section~\ref{isolated pairs} show that inserting diagonals of zeros might give other interesting examples of dual pairs. Would further insertion of zero diagonals preserve duality? Also, under which conditions is duality preserved via convex combinations of extreme points? Such questions could be considered in the study of the relation between Birkhoff and Alternating Sign Matrix polytopes, thereby contributing to other fields of research. We plan to investigate these points further in a follow-up work.

\blue{
\section*{Acknowledgments}
Toma\v{z} Ko\v{s}ir was supported in part by the Research and Innovation Agency of Slovenia (research core funding P1-0448 and research project N1-0217). 

The authors are particularly grateful to Michiel Hochstenbach for introducing them during a research visit of the first author with him at the Eindhoven University of Technology.

We would also like to express our gratitude to the reviewers for their careful reading, and thoughtful comments and suggestions on the original version of the manuscript.
}

\bibliographystyle{plain}
\bibliography{references}

\end{document}